\documentclass[11pt,reqno]{amsart}

\usepackage{amsmath}
\usepackage{amsmath}
\usepackage[margin=2.5cm]{geometry}
\usepackage{color}               
\usepackage{faktor}
\usepackage{tikz}
\usepackage{tikz-cd}
\usepackage{spverbatim}
\usepackage{float}
\usepackage[pdftex,hyperfootnotes=false,colorlinks=false,hypertexnames=false]{hyperref}

\usepackage{latexsym,enumitem}
\usepackage{amssymb}
\usepackage{epsfig}
\usepackage{psfrag}
\usepackage{amsthm}
\usepackage{amscd}
\usepackage{amsfonts}
\usepackage{graphics,caption}
\usepackage[all]{xy}
\usepackage{etoolbox}
\patchcmd{\quote}{\rightmargin}{\leftmargin 2em \rightmargin}{}{}
\usepackage{mathrsfs}
\usepackage{lipsum}

 \usepackage{mathtools}
\usepackage[normalem]{ulem}
\usepackage{array}
\usepackage{tikz-cd}
\usepackage{quiver}
\usepackage[dvipsnames]{xcolor}
\usepackage{colortbl}
\usepackage{todonotes}
\usepackage{float}

\usepackage[oldstyle]{libertine}%
\usepackage[T1]{fontenc}
\usepackage[final]{microtype}
\usepackage[utf8]{inputenc}
\usepackage{csquotes}
\makeatletter
\renewcommand*\libertine@figurestyle{LF}
\makeatother
\usepackage[libertine,libaltvw,liby]{newtxmath}
\makeatletter
\renewcommand*\libertine@figurestyle{OsF}
\makeatother
\usepackage[noerroretextools,backend=biber,style=alphabetic,maxalphanames=5,giveninits,sorting=nyt,%
maxbibnames=99,doi=false,isbn=false,url=false,eprint=false]{biblatex}
\usepackage{tikz-cd}
\usepackage[noabbrev,capitalise]{cleveref}
\usepackage{autonum}

\usepackage{todonotes}

\newcommand{\R}{\mathbb{R}} 
 
\newcommand{\N}{\mathbb{N}} 
\newcommand{\Z}{\mathbb{Z}} 
\newcommand{\Q}{\mathbb{Q}}

\newcommand{\rsph}{\mathbb{P}^1}

\newcommand{\ba}{\mathcal{A}}
\newcommand{\fb}{\mathcal{F}}

\newcommand{\gb}{\mathcal{G}}
\newcommand{\sba}{\mathcal{S}}
\newcommand{\vb}{\mathcal{V}}

\newcommand{\sub}{\subseteq}

\newcommand{\bt}{\mathbf{t}}
\newcommand{\bs}{\mathbf{s}}

\newcommand{\im}{\operatorname{Im}}
\newcommand{\aut}{\operatorname{Aut}}
\newcommand{\val}{\operatorname{val}}

\newcommand{\defeq}{\vcentcolon=}

\newcommand{\iwedge}{\bigwedge^{\frac{\infty}{2}}}

\theoremstyle{plain}
    \newtheorem{theorem}{Theorem}[section]
    \newtheorem{construction/theorem}[theorem]{Construction/Theorem}
    \newtheorem{corollary}[theorem]{Corollary}
    \newtheorem{lemma}[theorem]{Lemma}
    \newtheorem{proposition}[theorem]{Proposition}
\theoremstyle{definition}
    \newtheorem{remark}[theorem]{Remark}
    \newtheorem{example}[theorem]{Example}
    \newtheorem{definition}[theorem]{Definition}

\title{Tropicalising hypergeometric $\tau$-functions}
\author[M.~A.~Hahn]{Marvin Anas Hahn}
\address{M.~A.~Hahn: School of Mathematics 17, Westland Row, Trinity College Dublin, Dublin 2, Ireland}
\email{hahnma@tcd.ie}

\author[B.~O'Callaghan]{Brian O'Callaghan}
\address{B.~O'Callaghan: School of Mathematics 17, Westland Row, Trinity College Dublin, Dublin 2, Ireland}
\email{brocalla@tcd.ie}

\author[J.~Wahl]{Jonas Wahl}
\address{J.~Wahl: German Research Centre for Artificial Intelligence (DFKI), Berlin, Germany}
\email{jonas.wahl@dfki.de}


\begin{document}
\maketitle

\begin{abstract}
    Weighted Hurwitz numbers arise as coefficients in the power sum expansion of deformed hypergeometric $\tau$--functions. They specialise to essentially all known cases of Hurwitz numbers, including classical, monotone, strictly monotone and completed cycles Hurwitz numbers. In this work, we develop a tropical geometry framework for their study, thus enabling a simultaneous investigation of all these cases. We obtain a correspondence theorem expressing weighted Hurwitz numbers in terms of tropical covers. Using this tropical approach, we generalise most known structural results previously obtained for the aforementioned special cases to all weighted Hurwitz numbers. In particular, we study their polynomiality and derive wall--crossing formulae. Moreover, we introduce elliptic weighted Hurwitz numbers and derive tropical mirror symmetry for these new invariants, i.e. we prove that their generating function is quasimodular and that they may be expressed as Feynman integrals.
\end{abstract}

\section{Introduction}
Double Hurwitz numbers $H_g(\mu,\nu)$ count branched genus $g$ covers of $\mathbb{P}^1$ with prescribed ramification profiles $\mu$ and $\nu$ over $0$ and $\infty$, and simple ramification over $r$ additional branch points. In seminal work \cite{okounkov2000todaequationshurwitznumbers}, Okounkov proved that the generating function of double Hurwitz numbers in a double symmetric power sum expansion is a $\tau$--function of the 2D Toda lattice. Ever since, the same has been proved for many related enumerative invariants, e.g. monotone Hurwitz numbers \cite{goulden2016} and strictly monotone Hurwitz numbers \cite{alexandrov2016ramifications}, as well as completed cycles Hurwitz numbers -- an important special case of descendant Gromov--Witten invariants \cite{okounkov2006gromov}.

In \cite{orlov2001hypergeometric} Orlov and Scherbin introduced the notion of \textbf{hypergeometric $\tau$--functions of the 2D Toda lattice}. The generating functions of virtually all known enumerative invariants of $\mathbb{P}^1$, including all variants of Hurwitz numbers mentioned above, with two relative conditions prescribed by partitions $\mu$ and $\nu$ yield hypergeometric $\tau$--functions of the Toda lattice.\footnote{The recently introduced \textbf{leaky Hurwitz numbers} \cite{cavalieri2025pluricanonical} provide a distinguished exception.} It is therefore natural to ask, given a hypergeometric $\tau$--function, do the coefficients in its double symmetric power sum expansion have some enumerative meaning? There is a plethora of literature on this question -- see \cite{harnad2016weighted} for an overview -- and the general answer is affirmative. In particular, in \cite{guaypaquet20152dtodataufunctionscombinatorial} a general method for interpreting the coefficients of $\tau$--functions of hypergeometric type as enumerations akin to Hurwitz numbers was introduced. Ultimately, this led to the introduction of \textbf{weighted Hurwitz numbers} $H_g^G(\mu,\nu)$ ($H_g^{\tilde{G}}(\mu,\nu)$) which depend of a \textbf{weight generating function} $G(z)$ (resp. its dual $\tilde{G}(z)$). Weighted Hurwitz numbers provide a unified framework for virtually all Hurwitz--type counts by specifying the weighted generating function $G$ or $\tilde{G}$. For example for $G(z)=e^z$ one recovers classical double Hurwitz numbers, while $\tilde{G}(z)=\frac{1}{1-z}$ and $G(z)=1+z$ yield monotone and strictly monotone Hurwitz numbers. For $\mathcal{S}(z)=2\mathrm{sinh(z/2)}/z$ and $G(z)=\mathrm{exp}(\mathcal{S}(\partial_{z})\frac{z^r}{r!})$ we recover completed cycles Hurwitz numbers.

\subsection{Structural properties}
Hurwitz numbers are known to satisfy many interesting properties. In seminal work \cite{goulden2003geometrydoublehurwitznumbers}, Goulden, Jackson and Vakil proved that for fixed $g$ and length of $\mu$ and $\nu$, double Hurwitz numbers $H_g(\mu,\nu)$ are piecewise polynomial in the entries of $\mu$ and $\nu$. The regions of polynomiality are the complement of the resonance arrangement. In a series of works \cite{shadrin2008chamber,cavalieri2010chamberstructuredoublehurwitz,johnson2010doublehurwitznumbersinfinite} this polynomiality was further investigated. In particular, it was proved that the difference of polynomials in adjacent regions of polynomiality may be expressed in terms of double Hurwitz numbers with smaller input data. This is called a \textbf{wall--crossing formula}. Analogous results for completed cycles Hurwitz numbers were proved in \cite{Shadrin_2012}. Further extensions to monotone and strictly monotone Hurwitz numbers were proved to be true in \cite{goulden2016,hahn2018wall,hahn2019tropicaljucyscovers,hahn2019monodromy,hahn2019wallcrossingrecursionformulaetropical}. We wish to highlight the \textbf{tropical geometry} approach employed in \cite{cavalieri2010chamberstructuredoublehurwitz}. Tropical geometry is a relatively new field of mathematics that bridges algebraic geometry and combinatorics. In \cite{cavalieri2010tropical}, a tropical interpretation of double Hurwitz numbers was derived, i.e. double Hurwitz numbers were expressed a weighted enumeration of maps between graphs. These maps are called \textbf{tropical covers}. The piecewise polynomiality of double Hurwitz numbers is almost immediate from this combinatorial interpretation in conjunction with Ehrhart theory. A careful combinatorial analysis reveals the wall--crossing behaviour. A similar tropical approach was used in \cite{hahn2019monodromy,hahn2019tropicaljucyscovers,hahn2019wallcrossingrecursionformulaetropical} to extend these results for monotone and strictly monotone Hurwitz numbers.

Another important property of Hurwitz numbers arises when one studies branched coverings of elliptic curves instead of $\mathbb{P}^1$. The resulting analogous enumerations are called \textbf{elliptic Hurwitz numbers}. As elliptic curves are famously the simplest Calabi--Yau manifolds, elliptic Hurwitz numbers have been studied in the context of mirror symmetry. Two seminal results were obtained by Dijkgraaf in the 1990s \cite{dijkgraaf1995mirror}, where he confirmed two predictions of mirror symmetry taking shape as structural properties of Hurwitz numbers. In particular, he proved that elliptic Hurwitz numbers may be expressed as Feynman integrals and their generating function is a quasimodular form. Much attention has been devoted to refining the mathematics surrounding these two results, see e.g. \cite{10.1007/978-1-4612-4264-2_6,bloch2000character}. In \cite{B_hm_2015,goujard2017countingfeynmanlikegraphsquasimodularity} a tropical geometry approach was developed. It was observed that the Feynman integrals used in Dijkgraaf's work naturally arise from combinatorial types of tropical covers. Moreover, the quasimodularity could be refined as it was proved that indeed the contribution of each combinatorial of tropical covers is quasimodular. These results were extended to elliptic completed cycles Hurwitz numbers in \cite{tropmirror2022}.  In \cite{Hahn_2022} it was proved that the quasimodularity and refined quasimodularity extends to monotone and strictly monotone Hurwitz numbers as well.

\subsection{Aims and contributions}
The aim of this work is to introduce a method for the study enumerative invariants arising from $\tau$--functions of hypergeometric type via tropical geometry. In our first main result in \cref{thm-tropcorr}, we obtain a tropical correspondence theorem for weighted double Hurwitz numbers $H^G_g(\mu,\nu)$ and $H^{\tilde{G}}_g(\mu,\nu)$. Our strategy is to employ the cut--and--join operator for weighted Hurwitz numbers in the Fock space formalism obtained in \cite{alexandrov2018fermionic}. It is well--known that an appropriate version of Wick's theorem \cite{wick1950evaluation} allows to express vacuum expecation in the bosonic Fock space via tropical combinatorics, see e.g. \cite{block2016refined,cavalieri2021counting}. The same applies here; a careful re--writing of the cut--and--join operator allows an expression of weighted Hurwitz numbers via tropical covers. Quite interestingly, the class of covers we obtain here is the same as the one obtained in \cite{hahn2019tropicaljucyscovers} for monotone and strictly monotone double Hurwitz numbers. The difference is that we count the same covers with a different weighting depending on the choice of weight generating function. Moreover, we define weighted elliptic Hurwitz numbers in \cref{sec-weightell} in analogy to the classical setting. We obtain a correspondence theorem for them as well in \cref{thm-ellcorr}.

Employing this tropical correspondence theorem, we show the above mentioned structural properties of classical, completed cycles, monotone and strictly monotone Hurwitz numbers extend to the entire weighted Hurwitz numbers problem. Indeed, we derive piecewise polynomiality in \cref{thm-piecewisepoly}) and wall--crossing formulae in \cref{wallcrossingthm}. Moreover, we prove quasimodularity of elliptic weighted Hurwitz numbers in \cref{thm-quasimod} and refined quasimodularity in \cref{thm-refinquasi}. Finally, we obtain an expression in terms of Feynman integrals in \cref{thm-feynman}. Therefore, we offer the point of view that in some sense the applicability of tropical geometry techniques to Hurwitz--type enumerations is founded in the integrable structure of these invariants.

\subsection{Structure}
We begin with a preliminary section on tropical geometry, the Fock space formalism, weighted Hurwitz numbers, hypergeometric $\tau$--functions and the theory surrounding quasimodular forms in \cref{sec-prelim}. We follow the motto "bosonification is tropicalisation" in \cref{sec-bostrop} to derive a tropical correspondence theorem for weighted double Hurwitz numbers and study their polynomiality properties in \cref{sec-poly}. Finally, we study elliptic weighted Hurwitz numbers in \cref{sec-tropmirr} and establish tropical mirror symmetry.

\subsection{Acknowledgements.} This project grew out of a Research in Paris stay of M.A.H. and J.W. at the Institut Henri Poincaré in 2019. We thank the institute for their hospitality. After laying dormant for many years, B.O. worked out the many missing details for his thesis under M.A.H.'s supervision. B.O.'s work was partially supported by the Hamilton Trust fund.

\section{Preliminaries}
\label{sec-prelim}
We briefly recall the necessary basics around tropical covers, Fock spaces, hypergeometric $\tau$-functions and (weighted) Hurwitz numbers.

\subsection{Tropical covers}
We begin by defining abstract tropical curves.

\begin{definition}
    An abstract tropical curve is a connected graph $\Gamma$ with the following data:
    \begin{enumerate}
        \item The vertex set is denoted by $V(\Gamma)$ and the edge set by $E(\Gamma)$.
        \item The $1$-valent vertices of $\Gamma$ are called \textit{leaves} and their set is denoted by $V^\infty(\Gamma)$. The remaining vertices are called \textit{inner vertices} and their set is denoted by $V^0(\Gamma)$.
        \item The edges adjacent to leaves are called \textit{ends} and their set is denoted by $E^\infty(\Gamma)$. The remaining edges are called \textit{bounded edges} with their set denoted by $E^0(\Gamma)$.
        \item There is a length function
        \begin{align}
            \ell\colon E(\Gamma)\to\mathbb{R}_{\ge0}\cup\{\infty\}\\
            e\mapsto\ell(e),
        \end{align}
        such that $\ell^{-1}(\infty)=E^\infty(\Gamma)$.
        \item There is a map
        \begin{align}
            g\colon V^0(\Gamma)\to\mathbb{Z}_{\ge0}.
        \end{align}
        We call $g(v)$ for $v\in V^0(\Gamma)$ the genus of $v$.
    \end{enumerate}
    Moreover, we define the genus of $\Gamma$ as $g(\Gamma)=b^1(\Gamma)+\sum_{v\in V^0(\Gamma)}g(v)$, where $b^1(\Gamma)$ is the first Betti number of $\Gamma$, i.e. $|V(\Gamma)|-|E(\Gamma)|=1-g$.
    An isomorphism of abstract tropical curves is a map of the underlying graph that respects the maps $\ell$ and $g$.

    When we drop the length function, we obtain the \textit{combinatorial type} of an abstract tropical curve $\Gamma$.
\end{definition}

\begin{definition}
    A tropical cover $\pi: \Gamma_1 \to \Gamma_2$ is a surjective harmonic map of abstract tropical curves, i.e.: 
\begin{itemize}
\itemsep0em 
    \item $\pi(V(\Gamma_1))\subset V(\Gamma_2)$;
    \item $\pi^{-1}(E(\Gamma_2)) \subset E(\Gamma_1)$;
    \item The function $\pi$ is piecewise integer affine linear. That is when we take any $e \in E(\Gamma_1)$, if we are to interpret $e$ and $\pi(e)$ as $[0,l(e)]$ and $[0, l(\pi(e))]$ respectively, then we must have that $\pi|_{e}$ is a bijective integer linear function, $[0,l(e)] \to [0, l(\pi(e))], t \mapsto \omega(e)\cdot t$, for $\omega(e) \in \Z$. If we have that $\pi(e) \in V(\Gamma_2)$ we set $\omega(e) = 0$. The quantity $\omega(e)$ is called the weight of $e$.
    \item $\pi$ is a harmonic map. That is, consider some $v \in V(\Gamma_2)$ and an edge $e$ adjacent to $v$. Then we define the local degree
    \[
    d_v = \sum_{\tilde{e} \mapsto e} \omega(\tilde{e})
    \]
    as the sum of weights of all edges $\tilde{e}$ mapping onto the chosen edge $e$. The harmonicity condition is that local degree at $v$ is well defined (independent of choice of $e$).
    This condition is also referred to as the balancing condition.
\end{itemize}
 We call two covers $\pi_1: \Gamma_1 \to \Gamma'$ and $\pi_2: \Gamma_2 \to \Gamma'$ equivalent if there exists an isomorphism of abstract tropical curves $g: \Gamma_1 \to \Gamma_2$ such that $\pi_2 \circ g = \pi_1$.
\end{definition}
We note that when all weights are strictly positive integers, then (after suitable refinement of $\Gamma_1$ and $\Gamma_2$) the image and preimage of vertices are precisely vertices.

\begin{definition}
    Let $\pi\colon\Gamma_1\to\Gamma_2$ a tropical cover. Let $v\in V(\Gamma_2)$, then we call $d=\sum_{w\in\pi^{-1}(v)}d_w$ the \textit{degree} of $\pi$. By the harmonicity condition, $d$ is independent of the choice of $v$.

    Moreover, for $e\in E(\Gamma_2)$, we define the ramification profile of $e$ as the partition $\mu_e$ of weights of edges in the preimage of $e$.
\end{definition}

In this article, we focus on tropical covers with two particular targest: the tropical projective line $\mathbb{P}^1_{\textit{trop}}$ and the tropical elliptic curve $E_{\textit{trop}}$.

\begin{definition}
    The tropical projective line is defined as
    \begin{equation}
         \mathbb{P}^1_{\textit{trop}}=\mathbb{R}\cup\{\pm\infty\},
    \end{equation}
    where we may add two valent vertices $v\in\mathbb{R}$. The tropical elliptic curve $E_{\textit{trop}}$ is the circle, where again we may add two valent vertices.
\end{definition}

The tropical projective line and tropical elliptic curve as illustrated in \cref{fig:linecircle}.

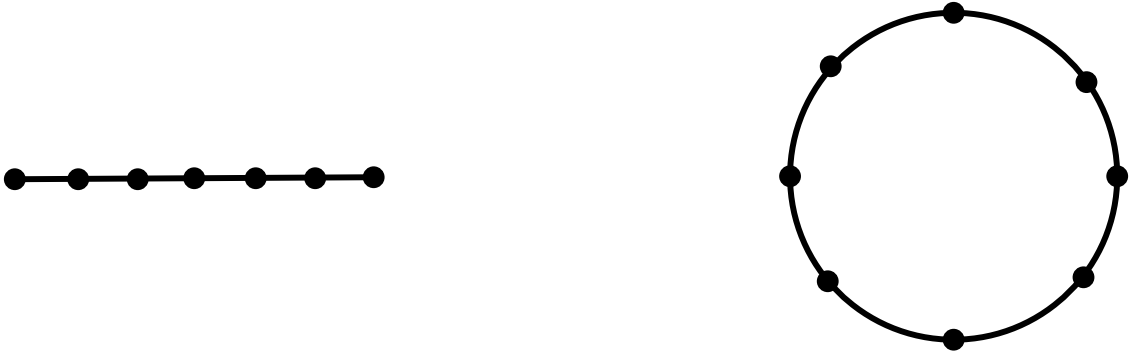
\begin{figure}
    \centering

\tikzset{every picture/.style={line width=0.75pt}} 

\begin{tikzpicture}[x=0.75pt,y=0.75pt,yscale=-1,xscale=1]

\draw [line width=2.25]    (29,120) -- (210,119) ;
\draw  [line width=2.25]  (420,118.5) .. controls (420,72.94) and (456.94,36) .. (502.5,36) .. controls (548.06,36) and (585,72.94) .. (585,118.5) .. controls (585,164.06) and (548.06,201) .. (502.5,201) .. controls (456.94,201) and (420,164.06) .. (420,118.5) -- cycle ;
\draw  [fill={rgb, 255:red, 0; green, 0; blue, 0 }  ,fill opacity=1 ] (24,120) .. controls (24,117.24) and (26.24,115) .. (29,115) .. controls (31.76,115) and (34,117.24) .. (34,120) .. controls (34,122.76) and (31.76,125) .. (29,125) .. controls (26.24,125) and (24,122.76) .. (24,120) -- cycle ;
\draw  [fill={rgb, 255:red, 0; green, 0; blue, 0 }  ,fill opacity=1 ] (205,119) .. controls (205,116.24) and (207.24,114) .. (210,114) .. controls (212.76,114) and (215,116.24) .. (215,119) .. controls (215,121.76) and (212.76,124) .. (210,124) .. controls (207.24,124) and (205,121.76) .. (205,119) -- cycle ;
\draw  [fill={rgb, 255:red, 0; green, 0; blue, 0 }  ,fill opacity=1 ] (56,120) .. controls (56,117.24) and (58.24,115) .. (61,115) .. controls (63.76,115) and (66,117.24) .. (66,120) .. controls (66,122.76) and (63.76,125) .. (61,125) .. controls (58.24,125) and (56,122.76) .. (56,120) -- cycle ;
\draw  [fill={rgb, 255:red, 0; green, 0; blue, 0 }  ,fill opacity=1 ] (86,120) .. controls (86,117.24) and (88.24,115) .. (91,115) .. controls (93.76,115) and (96,117.24) .. (96,120) .. controls (96,122.76) and (93.76,125) .. (91,125) .. controls (88.24,125) and (86,122.76) .. (86,120) -- cycle ;
\draw  [fill={rgb, 255:red, 0; green, 0; blue, 0 }  ,fill opacity=1 ] (114.5,119.5) .. controls (114.5,116.74) and (116.74,114.5) .. (119.5,114.5) .. controls (122.26,114.5) and (124.5,116.74) .. (124.5,119.5) .. controls (124.5,122.26) and (122.26,124.5) .. (119.5,124.5) .. controls (116.74,124.5) and (114.5,122.26) .. (114.5,119.5) -- cycle ;
\draw  [fill={rgb, 255:red, 0; green, 0; blue, 0 }  ,fill opacity=1 ] (145.5,119.5) .. controls (145.5,116.74) and (147.74,114.5) .. (150.5,114.5) .. controls (153.26,114.5) and (155.5,116.74) .. (155.5,119.5) .. controls (155.5,122.26) and (153.26,124.5) .. (150.5,124.5) .. controls (147.74,124.5) and (145.5,122.26) .. (145.5,119.5) -- cycle ;
\draw  [fill={rgb, 255:red, 0; green, 0; blue, 0 }  ,fill opacity=1 ] (175.5,119.5) .. controls (175.5,116.74) and (177.74,114.5) .. (180.5,114.5) .. controls (183.26,114.5) and (185.5,116.74) .. (185.5,119.5) .. controls (185.5,122.26) and (183.26,124.5) .. (180.5,124.5) .. controls (177.74,124.5) and (175.5,122.26) .. (175.5,119.5) -- cycle ;
\draw  [fill={rgb, 255:red, 0; green, 0; blue, 0 }  ,fill opacity=1 ] (415,118.5) .. controls (415,115.74) and (417.24,113.5) .. (420,113.5) .. controls (422.76,113.5) and (425,115.74) .. (425,118.5) .. controls (425,121.26) and (422.76,123.5) .. (420,123.5) .. controls (417.24,123.5) and (415,121.26) .. (415,118.5) -- cycle ;
\draw  [fill={rgb, 255:red, 0; green, 0; blue, 0 }  ,fill opacity=1 ] (580,118.5) .. controls (580,115.74) and (582.24,113.5) .. (585,113.5) .. controls (587.76,113.5) and (590,115.74) .. (590,118.5) .. controls (590,121.26) and (587.76,123.5) .. (585,123.5) .. controls (582.24,123.5) and (580,121.26) .. (580,118.5) -- cycle ;
\draw  [fill={rgb, 255:red, 0; green, 0; blue, 0 }  ,fill opacity=1 ] (434,171.5) .. controls (434,168.74) and (436.24,166.5) .. (439,166.5) .. controls (441.76,166.5) and (444,168.74) .. (444,171.5) .. controls (444,174.26) and (441.76,176.5) .. (439,176.5) .. controls (436.24,176.5) and (434,174.26) .. (434,171.5) -- cycle ;
\draw  [fill={rgb, 255:red, 0; green, 0; blue, 0 }  ,fill opacity=1 ] (563,169.5) .. controls (563,166.74) and (565.24,164.5) .. (568,164.5) .. controls (570.76,164.5) and (573,166.74) .. (573,169.5) .. controls (573,172.26) and (570.76,174.5) .. (568,174.5) .. controls (565.24,174.5) and (563,172.26) .. (563,169.5) -- cycle ;
\draw  [fill={rgb, 255:red, 0; green, 0; blue, 0 }  ,fill opacity=1 ] (497.5,201) .. controls (497.5,198.24) and (499.74,196) .. (502.5,196) .. controls (505.26,196) and (507.5,198.24) .. (507.5,201) .. controls (507.5,203.76) and (505.26,206) .. (502.5,206) .. controls (499.74,206) and (497.5,203.76) .. (497.5,201) -- cycle ;
\draw  [fill={rgb, 255:red, 0; green, 0; blue, 0 }  ,fill opacity=1 ] (435.5,63) .. controls (435.5,60.24) and (437.74,58) .. (440.5,58) .. controls (443.26,58) and (445.5,60.24) .. (445.5,63) .. controls (445.5,65.76) and (443.26,68) .. (440.5,68) .. controls (437.74,68) and (435.5,65.76) .. (435.5,63) -- cycle ;
\draw  [fill={rgb, 255:red, 0; green, 0; blue, 0 }  ,fill opacity=1 ] (564.5,71) .. controls (564.5,68.24) and (566.74,66) .. (569.5,66) .. controls (572.26,66) and (574.5,68.24) .. (574.5,71) .. controls (574.5,73.76) and (572.26,76) .. (569.5,76) .. controls (566.74,76) and (564.5,73.76) .. (564.5,71) -- cycle ;
\draw  [fill={rgb, 255:red, 0; green, 0; blue, 0 }  ,fill opacity=1 ] (497.5,36) .. controls (497.5,33.24) and (499.74,31) .. (502.5,31) .. controls (505.26,31) and (507.5,33.24) .. (507.5,36) .. controls (507.5,38.76) and (505.26,41) .. (502.5,41) .. controls (499.74,41) and (497.5,38.76) .. (497.5,36) -- cycle ;

\end{tikzpicture}
\caption{An illustration of the tropical projective line $\mathbb{P}^1_{\textit{trop}}$ on the left and the tropical elliptic curve $E_{\textit{trop}}$ on the right.}
    \label{fig:linecircle}
\end{figure}

\subsection{Semi--infinite wedge formalism}
We now move to the basic notions surrounding the semi-infinite wedge formalism. First, we introduce some notation. Let $\varsigma(z)=2\mathrm{sinh}(z/2)$ and $\mathcal{S}(z)=\varsigma(z)/z$. We recall that
\begin{equation}
    \frac{1}{\varsigma(z)}=\sum_{g\ge0}c_{2g-1}z^{2g-1}
\end{equation}
with $c_{-1}=1$ and for $g\ge0$ we have
\begin{equation}
    c_{2g-1}=-\frac{(2^{2g-1}-1)B_{2g}}{2^{2g-1}(2g)!}=(-1)^g\int_{\overline{\mathcal{M}}_{g,1}}\lambda_g\psi_1^{2l-2}=(-1)^g\langle\tau_{2g-2(\omega)}\rangle_{g,1}^{\mathbb{P}^1},
\end{equation}
where
\begin{enumerate}
    \item $B_k$ is the $k$-th Bernoulli number.
    \item $\overline{\mathcal M}_{g,n}$ is the moduli space of stable genus $g$ curves with $n$ marked points.
    \item $\lambda_g$ is the top Chern class of the Hodge bundle $\mathbb{E}\to\overline{\mathcal M}_{g,n}$.
    \item $\psi_i$ is the first Chern class of the $i$--th cotangent bundle $\mathbb{L}\to\overline{\mathcal M}_{g,n}$ whose fiber at $(C,p_1,\dots,p_n)$ is the cotangent space of $C$ at $p_i$.
\end{enumerate}

Moreover, for two partitions $\mu$, $\nu$ of a positive integer $d$ and integers $k_1,\dots,k_n\ge0$ with $\sum k_i=2g+2d-2$, we denote
\begin{equation}
    \langle\mu| \tau_{k_1}(\omega)\cdots\tau_{k_1}(\omega)|\nu\rangle_{g,n}^{\mathbb{P}^1}=\int_{[\mathcal{M}_{g,n}(\mathbb{P}^1,\mu,\nu)]^{\textit{vir}}}\prod\mathrm{ev}_i^\ast(\omega)\psi_i^{k_i},
\end{equation}

where $\mathcal{M}_{g,n}(\mathbb{P}^1,\mu,\nu)$ is the moduli space of stable maps to $\mathbb{P}^1$ relative to $\mu$ and $\nu$. Moreover, the $i$--th evaluation map $\mathrm{ev}_i\colon\mathcal{M}_{g,n}(\mathbb{P}^1,\mu,\nu)\to\mathbb{P}^1$ is defined as $\mathrm{ev}_i((C,p_1,\dots,p_n,f))=f(p_i)$ for $f\colon C\to\mathbb{P}^1$ the relative stable map. Further, $\omega$ denotes the point class in $\mathbb{P}^1$. For more details, we refer to \cite{vakil2006modulispacecurvesgromovwitten}. 

We also introduce superscripts $\circ$ and $\bullet$ to specify whether we deal with \textit{connected} or \textit{non-necessarily connected} (also \textit{disconnected}) curves respectively.

Now, let $V=\bigoplus_{i \in \Z } \underline{i + \frac{1}{2}}$ and define the semi--infinite wedge space

\begin{equation}
    \iwedge V \defeq \bigoplus_{(i_k)}\underline{i_1} \wedge \underline{i_2} \wedge \cdots
\end{equation}
summing over all decreasing sequences of half integers such that $i_k + k - 1/2 = c$ for sufficiently large $k$ for some integer $c$ which we call the \textit{charge}. We also have an inner product on $\iwedge V$ by defining the basis vectors $v_S$ to be orthonormal.

We will in particular care about the charge $0$ sector $\mathcal{V}_0$ defined as
\begin{equation}
    \mathcal{V}_0=\bigoplus_{n\in\mathbb{N}}\bigoplus_{\lambda\vdash n}\mathcal{C} \lambda,
\end{equation}
where $v_\lambda=\lambda_1-1/2\wedge\lambda_2-3/2\wedge\lambda_3-5/2\wedge\dots$ for integer partitions $\lambda$. We note the convention that partitions here are semi-infinite with all $0$s at the end. On $\mathcal{V}_0$ we have a natural inner product $\left< - \mid - \right>$ by imposing the basis elements to be orthonormal.

The vector $v_\emptyset$ corresponding to the empty partition is called the \textit{vacuum vector} and we denote it by $|0\rangle$. We further construct the covacuum vector in the dual $\vb_0^*$, which we denote $\left< 0 \right|$. 
For any operator $\mathcal P$ acting on $\vb_0$, we define the vacuum expectation $\left< \mathcal P \right>^\bullet \defeq \left< 0 \mid \mathcal P \mid 0 \right>^\bullet$. 
We construct operators $\psi_k$ for every half integer $k \in \Z + \frac 12$, $\psi_k: \vb_ 0 \to \vb_0, (\underline{i_1} \wedge \underline{i_2} \wedge \dots) \mapsto (\underline{k} \wedge \underline{i_1} \wedge \underline{i_2} \wedge \dots) $, and let $\psi_k^\dagger$ be the adjoint operator with respect to the inner product $\left< - \mid - \right>$ . 

We may define operators on $\mathcal{V}_0$ via normally ordered products of $\psi$ operators as follows:   
\[
E_{i,j} \defeq \ : \hspace{-.2em} \psi_i \psi_j^\dagger \hspace{-.3em}: \ = \begin{cases}
    \psi_i \psi_j^\dagger, \quad \, \, \, \, j>0, \\
    -\psi_j^\dagger \psi_i, \quad j<0.
\end{cases}
\]

For any integer $n$ and a formal variable $z$ we may then obtain the operators
\[
\mathcal{E}_n(z) \defeq \sum_{k \in \Z + \frac 12} e^{z(k - \frac n2 )} E_{k-n,k} + \frac{\delta_{n,0}}{\varsigma(z)}, \quad \tilde{\mathcal E}_0 \defeq \sum_{k \in \Z + \frac 12} e^{zk} E_{k,k}, \quad \alpha_n \defeq \mathcal{E}_n(0) = \sum_{k \in \Z + \frac 12} E_{k-n,k}.
\]
In particular, we have the commutation relation 
\[
[ \alpha_k, \alpha_l] = k \delta_{k+l,0}.
\]
We further extract the \textbf{completed cycles operators} $\fb_k$ from $\mathcal{E}_0$ as
\[
\fb_k \defeq [z^k] \tilde{\mathcal{E}}_0(z) = \sum_{l \in \Z + \frac 12} \frac{l^k}{k!} E_{l,l}.
\]

We have the following expansion for the operators $\mathcal{F}_k$.

\begin{proposition}[\cite{hahn2019tropicaljucyscovers}]
\[
\fb_k = \sum_{g=0}^\infty \sum_{\mathbf{x} \in S \Z^{k+1-2g}} 
\left< \mathbf x^+ \mid \tau_{k-1} \mid  \mathbf x^- \right>_g^{\rsph, \circ} 
    \prod_{0>x_i \in \mathbf{x}} \alpha_{x_i} \prod_{0 < x_j \in \mathbf{x}} \alpha_{x_j}.
\]
\end{proposition}

We denote $\fb_0 = C$ the charge operators and $\fb_1 = E$ the energy operator. 

\begin{remark}
    We recall that the operators $\mathcal{F}_k$ were famously introduced by Okounkov and Pandharipande in \cite{okounkov2006gromov} where they proved the \textbf{Gromov--Witten/Hurwitz correspondence}
    \begin{equation}
        \langle\mu| \tau_{k_1}(\omega)\cdots\tau_{k_1}(\omega)|\nu\rangle_{g,n}^{\mathbb{P}^1}=\langle\alpha_\mu\mid\mathcal{F}_{k_1+1}\cdots\mathcal{F}_{k_n+1}\mid\alpha_{-\nu}\rangle,
    \end{equation}
    where $\alpha_\mu=\prod \alpha_{\mu_i}$ and $\alpha_{-\nu}=\prod\alpha_{-\nu_j}$.
\end{remark}

As a next step, we may now define a family of operators $\gb_l$ via
\[
\frac{\tilde{\mathcal{E}}_0(z)}{\varsigma(z)} - E  = \sum_{l \geq 1} \frac{z^l}{l!} \gb_{l+1},
\]
which were first introduced in \cite{hahn2019tropicaljucyscovers} in the context of (strictly) monotone Hurwitz numbers. They admit the following expansion

\[
\label{equ-gviafoper}
\gb_{l+1} \defeq l! \sum_{k=0}^\infty c_{2k-1}\fb_{l-(2k-1)}. 
\]

By \cite{hahn2019tropicaljucyscovers}, we have the following.

\begin{lemma}\label{gbasgromovwitten}
\[
\gb_{l+1} = l! \hspace{-2em}\sum_{\substack{g_1,g_2 = 0 \\ \mathbf{x} \in S\Z^{l+2-2g_1-2g_2}}} \hspace{-1.4em} \Big< \tau_{2g_2 -2}(\omega) \Big>_{g_2}^{\rsph, \circ} \Big< \mathbf{x}^+ \mid \tau_{2g_1 -2+ \ell(\mathbf{x}^+) + \ell(\mathbf{x}^-)}(\omega) \mid \mathbf{x}^-   \Big>_{g_1}^{\rsph, \circ} \prod_{0 > x_i \in \mathbf{x}} \alpha_{x_i} \prod_{0 < x_j \in \mathbf{x}} \alpha_{x_j}.
\]
\end{lemma}

As derived in \cite{okounkov2006gromov}, the coefficients of this expansion have a closed expression 

\begin{theorem}\label{polynomialgw}
\begin{equation}
\left< \mathbf x^+ \mid  \tau_{2g - 2 + \ell( \mathbf x^+ ) + \ell(\mathbf{x}^-)} \mid \mathbf x^- \right>_g^{\rsph, \circ} = \frac{1}{|\aut(\mathbf x^+)| | \aut(\mathbf{x}^-)|}[z^{2g}]\frac{\prod_{\mathbf{x}_i^+} \mathcal S (\mathbf{x}_i z) \prod_{\mathbf{x}_i^-} \mathcal S (\mathbf{x}_i z)}{\mathcal{S}(z)}.
\end{equation}
\end{theorem}

\subsection{(Weighted) Hurwitz numbers and hypergeometric $\tau$--functions}
In this section, we introduce the basics of Hurwitz numbers and explain their relation to hypergeometric $\tau$--functions.

We begin with the definition of \textbf{double Hurwitz numbers}.

\begin{definition}
\label{def-doublhur}
    Let $g\ge0$ be a non--negative integer, $d$ a positive integer and consider two partitions $\mu,\nu$ of $d$. Let $r=2g-2+\ell(\mu)+\ell(\nu)$ and fix $p_1,\dots,p_r\in\mathbb{P}^1$ pairwise distinct points. We consider \textbf{Hurwitz covers of type $(g,\mu,\nu)$}, i.e. holomorphic maps $f\colon S\to\mathbb{P}^1$, such that
    \begin{itemize}
        \item $S$ is a compact genus $g$ Riemann surface,
        \item $f$ has ramification profile $\mu$ over $0$, $\nu$ over $\infty$ and $(2,1\dots,1)$ over $p_i$,
        \item $f$ is unramified else.
    \end{itemize}
    We call two Hurwitz covers $f_1\colon S_1\to \mathbb{P}^1$ and $f_2\colon S_2\to\mathbb{P}^1$ equivalent if there exists an isomorphic $g\colon S_1\to S_2$ with $f_1=f_2\circ g$.

    Then, we define double Hurwitz numbers
    \begin{equation}
        H_g^\bullet(\mu,\nu)=\sum_{[f]}\frac{1}{|\mathrm{Aut}(f)|},
    \end{equation}
    where the sum is taken over all equivalence classes of Hurwitz covers of type $(g,\mu,\nu)$.

    When we additionally impose that $S$ is connected, we obtain connected double Hurwitz numbers $H_g^\circ(\mu,\nu)$.
\end{definition}

The following result was proved by Hurwitz in \cite{Hurwitz1891}. We note that for a permutation $\sigma\in S_d$, we denote the partition denoting its conjugacy class by $C(\sigma)$.

\begin{theorem}
\label{thm-hurw}
    In the same set-up as \cref{def-doublhur}, we call $(\sigma_1,\tau_1,\dots,\tau_r,\sigma_2)$ a factorisation of type $(g,\mu,\nu)$ if
    \begin{itemize}
        \item $\sigma_i,\tau_j\in S_d$,
        \item $C(\sigma_1)=\mu$, $C(\sigma_2)=\nu$,
        \item $C(\tau_i)=(2,1,\dots,1)$, i.e. $\tau_i$ are transpositions,
        \item $\tau_r\cdots\tau_1\sigma_1=\sigma_2$.
    \end{itemize}
    Denoting the set of factorisations of type $(g,\mu,\nu)$ by $\mathcal{F}_g(\mu,\nu)$, we have
    \begin{equation}
        H_g^\bullet(\mu,\nu)=\frac{1}{d!}|\mathcal{F}_g(\mu,\nu)|.
    \end{equation}
    We obtain the connected double Hurwitz numbers by additionally imposing that the subgroup of $S_d$ generated by $\sigma_1,\sigma_2,\tau_1,\dots,\tau_r$ is transitive.
\end{theorem}

\begin{remark}
    In the literature also relevant for the present work, there are two important variations of this factorisation problem. Expression $\tau_i=(a_i\,b_i)$ with $a_i<b_i$, when we additionally impose $b_i\le b_{i+1}$ we obtain \textbf{(weakly) monotone double Hurwitz numbers}, while for $b_i<b_{i+1}$ we obtain \textbf{strictly monotone double Hurwitz numbers.}
\end{remark}

We now want to package double Hurwitz numbers in a generating function. For this, we consider two infinite sets of variables $\mathbf{t}=(t_1,t_2,\dots),\mathbf{s}=(s_1,s_2,\dots)$. Moreover, we will need the following five families of symmetric functions associated to partions $\lambda$:
\begin{itemize}
    \item $s_\lambda(\mathbf{t})$ the Schur function,
    \item $e_\lambda(\mathbf{t})$ the elementary symmetric function,
    \item $h_\lambda(\mathbf{t})$ the homogeneous symmetric function,
    \item $p_\lambda(\mathbf{t})$ the symmetric power sum,
    \item $m_\lambda(\mathbf{t})$ the monomial symmetric function.
\end{itemize}

We now define a formal variable $\beta$ in the generating function of double Hurwitz numbers,
\begin{equation}    \tau(\mathbf{t},\mathbf{s},\beta)=\sum_{g=0}^\infty\sum_{d=0}^\infty\sum_{\mu,\nu\vdash d}\beta^{2g-2+\ell(\mu)+\ell(\nu)}\frac{H^\bullet_g(\mu,\nu)}{(2g-2+\ell(\mu)+\ell(\nu))!}p_\mu(\mathbf{t})p_\nu(\mathbf{s}).
\end{equation}

It is a general property of generating functions that $\mathrm{log}(\tau(\mathbf{t},\mathbf{s},\beta))$ is the analogous generating function of the connected double Hurwitz numbers $H^\circ_g(\mu,\nu)$.

Recall the definition of the shifted symmetric polynomial $f_2$ evaluated at a partition $\lambda$,
\begin{equation}
    f_2(\lambda)=\frac{1}{2}\sum_i\left[\left(\lambda_i-1+\frac{1}{2}\right)^2-\left(-i+\frac{1}{2}\right)^2\right].
\end{equation}
It is then not difficult to see (see e.g. \cite[section 2]{okounkov2000todaequationshurwitznumbers}) that using Burnside's formula, we obtain
\begin{equation}
\label{equ:okounkov}
        \tau(\mathbf{t},\mathbf{s},\beta)=\sum_{\lambda}e^{\beta f_2(\lambda)}s_\lambda(\mathbf{t})s_\lambda(\mathbf{s}).
\end{equation}

From this expression, Okounkov goes on to derive in \cite{okounkov2000todaequationshurwitznumbers} that $\tau(\mathbf{t},\mathbf{s},\beta)$ is a ($\beta$--deformation of a) $\tau$--function of an integrable hierarchy called the 2D Toda lattice. In fact, it is a member of the important family of \textbf{hypergeometric $\tau$--functions} introduced in \cite{orlov2001hypergeometric}. We give an ad-hoc definition of this family.\footnote{Note that technically, we restrict on the $N=0$ site of the Toda lattice.} Let $\rho\colon\mathbb{Z}\to\mathbb{C}^\ast$ be any map, then we define for any partition $\lambda$ the associated content product
\begin{equation}
    r_\lambda=\prod_{(i,j)\in\lambda} r_{j-i}\quad\textrm{for}\quad  r_i=\frac{\rho_i}{\rho_{i-1}}.
\end{equation}

These content products $r_{\lambda}$ turn out to be eigenvalues of diagonal operators in the fermionic Fock space. It may then be seen (see e.g. \cite{orlov2001hypergeometric}) that the function in the two infinite families of indeterminants $\mathbf{t}=(t_1,t_2,\dots)$ and $\mathbf{s}=(s_1,s_2,\dots)$ defined as
\begin{equation}    \tau_{\rho}(\mathbf{t},\mathbf{s})=\sum_{\lambda}r_\lambda s_\lambda(\mathbf{t})s_{\lambda}(\mathbf{s}),
\end{equation}
is a $\tau$--function of the 2D Toda lattice. These solutions are called \textit{hypergeometric $\tau$--functions} or Orlov--Scherbin $\tau$--functions.

In order to relate this to Hurwitz numbers, we introduce a one paramater family of $\tau$--functions. For this, we start by defining \textbf{weight generating functions} for a tuple of formal variables $\mathbf{c}=(c_1,c_2,\cdots)$

\begin{equation}
    G(z)=\prod_{i=1}^\infty(1+c_iz)\quad\textrm{and}\quad \tilde{G}(z)=\prod_{i=1}^\infty\frac{1}{1-c_iz},
\end{equation}
which we use to define the $\beta$--deformed content products

\begin{equation}
    r_\lambda^{(G,\beta)}=\prod_{(i,j)\in\lambda}G(\beta(j-i))\quad\textrm{and}\quad   r_\lambda^{(\tilde{G},\beta)}=\prod_{(i,j)\in\lambda}\tilde{G}(\beta(j-i)).
\end{equation}
From this, we finally define our $\beta$--deformed $\tau$--functions as

\begin{equation}
    \tau^{(G,\beta)}(\mathbf{t},\mathbf{s})=\sum_{\lambda}r_\lambda^{(G,\beta)}s_\lambda(\mathbf{t})s_\lambda(\mathbf{s})\quad\textrm{and}\quad\tau^{(\tilde{G},\beta)}(\mathbf{t},\mathbf{s})=\sum_{\lambda}r_\lambda^{(\tilde{G},\beta)}s_\lambda(\mathbf{t})s_\lambda(\mathbf{s}).
\end{equation}

We may now expand these $\tau$--functions as

\begin{equation}
        \tau^{(G,\beta)}(\mathbf{t},\mathbf{s})=\sum_{g=0}^\infty\sum_{d=0}^\infty\sum_{\mu,\nu\vdash d}\beta^{2g-2+\ell(\mu)+\ell(\nu)}\frac{H^{G,\bullet}_g(\mu,\nu)}{(2g-2+\ell(\mu)+\ell(\nu))!}p_\mu(\mathbf{t})p_\nu(\mathbf{s})
\end{equation}
where for now the numbers $H^{G,\bullet}_g(\mu,\nu)$ are implicitly defined via this expansion. However, it was shown in \cite{Guay_Paquet_2017} they have an interpretation as enumerative invariants called \textbf{weighted Hurwitz numbers}. There is a topological meaning in the sense of \cref{def-doublhur} as well as one in the setting of factorisations as in \cref{def-doublhur}. Here we only state the interpretation via the symmetric group and refer the interested reader to \cite[Section 1]{Guay_Paquet_2017}. To begin with, we need the following definition.

\begin{definition}
    Consider a tuple of transposition $(\tau_1,\dots,\tau_r)$ with $\tau_i=(a_i\,b_i)$ with $a_i<b_i$. We assign to it a \textbf{signature} $\lambda$. This is a partition of $r$ with parts equal to the number of transpositions $(a_i\,b_i)$ with the same value for $b_i$.

    We define by $m_{\mu\nu}^\lambda$ the number of factorisations $(\tau_1,\dots,\tau_r)$ of type $(g,\mu,\nu)$ with $(\tau_1,\dots,\tau_r)$ of signature $\lambda$ and $b_i\le b_{i+1}$. We call the condition $b_i\le b_{i+1}$ \textbf{monotonicity}.
\end{definition}

Moreover, we define for a partition $\lambda$

\begin{equation}
    G_\lambda(z)=e_\lambda(\mathbf{c})\quad\textrm{and}\quad\tilde{G}_\lambda(z)=h_\lambda(\mathbf{c}).
\end{equation}
Note that for $G(z)=\sum_{n=0}^\infty G_nz^n$ and $\tilde{G}(z)=\sum_{n=0}^\infty\tilde{G}_nz^n$, we have $G_\lambda=\prod G_{\lambda_i}$ and $\tilde{G}_\lambda=\prod \tilde{G}_{\lambda_i}$.

Finally, we obtain the following.

\begin{theorem}[{\cite{Guay_Paquet_2017}}]
    Let $g\ge0$ be a non--negative integer and $d$ a positive integer. Consider two partitions $\mu,\nu$ of $d$ and let $r=2g-2+\ell(\mu)+\ell(\nu)$, then we have
    \begin{equation}
        H_g^{G,\bullet}(\mu,\nu)=\frac{r!}{d!}\sum_{|\lambda|\vdash r}G_\lambda m_{\mu\nu}^\lambda
    \end{equation}
    and the analogous result for $\tilde{G}$.
\end{theorem}

\begin{remark}
    We note that by \cite{Guay_Paquet_2017} we have for $\lambda=(\lambda_1,\dots,\lambda_s)$ a partition of $r$ that
    \begin{equation}
        \widetilde{m}_{\mu\nu}^\lambda=\binom{r}{\lambda_1,\dots,\lambda_s}m_{\mu\nu}^\lambda,
    \end{equation}
    where by $\widetilde{m}_{\mu\nu}$ we denote the number of factorisation in $\mathcal{F}_g(\mu,\nu)$ with signature $\lambda$, i.e. where we drop the monotonicity condition. We claim that $G(z)=e^z$ recovers classical double Hurwitz numbers. This is worked out in detail in \cite{Guay_Paquet_2017}, but we wish to recall a few aspects.

    To begin with $G(z)=e^z$ is an extremal case, since it is a limit of weight generating functions in the sense of 
    \begin{equation}
        e^z=\lim_{n\to\infty}\left(1+\frac{z}{n}\right)^n.
    \end{equation}
    However, it still makes sense to define $G_\lambda$ via the power series expansion and we obtain $G_\lambda=\frac{1}{r!}\binom{r}{\lambda_1,\dots,\lambda_s}$.

    Thus, we obtain 

    \begin{equation}
        H_g^{e^z,\bullet}(\mu,\nu)=\frac{r!}{d!}\sum_{|\lambda|\vdash r}G_\lambda m_{\mu\nu}^\lambda=\frac{r!}{d!}\sum_{|\lambda|\vdash r}\frac{1}{r!}\widetilde{m}_{\mu\nu}^\lambda=\frac{1}{d!}\sum_{|\lambda|\vdash r}\widetilde{m}_{\mu\nu}^\lambda=\frac{1}{d!}|\mathcal{F}_g(\mu,\nu)|=H_g^\bullet(\mu,\nu)
    \end{equation}
    as desired.
\end{remark}

\subsection{Shifted symmetric functions and quasimodular forms}
We briefly introduce the basic notions around shifted symmetric functions and quasimodular forms. Let $\lambda$ be a partition of $d$. We denote by $\chi_\lambda$ the character of the representation of $S_d$ indexed by $\lambda$. For $|\lambda|=|\nu|$, we define the central character of a conjugacy class $C_\nu$ for a partition $\nu$ as
\begin{equation}
   f_\nu(\lambda)=|C_\nu|\frac{\chi^\lambda(C_\nu)}{\mathrm{dim}(\lambda)}, 
\end{equation}
where $|C_\nu|$ denotes the size of the conjugacy class in $S_|\nu|$ and $\mathrm{dim}(\lambda)$ denotes the size of the corresponding representation.

Extending this to the case of $|\lambda|=|\nu|$, we define

\begin{equation}
    f_\nu(\lambda)=\binom{|\lambda|}{|\nu|}|C_\nu|\frac{\chi^\lambda(\nu)}{\mathrm{dim}(\lambda)}.
\end{equation}

Central characters are important examples of the more general notion of \textbf{shifted symmetric functions}. A polynomial in variables $x_1,x_2,x_3,\dots$ is called shifted symmetric if it is symmetric under the permutation $x_i-i\mapsto x_j-j$. Indeed it proved by Okounkov and Olshanski \cite{zbMATH01117961} that $f_\nu$ is a shifted symmetric function in the entries of $\lambda$.

Denoting the algebra of shifted symmetric functions by $\Lambda^\ast$, we note that a distinguished generating set is given by the shifted symmetric power sums

\begin{equation}
    p_i^\ast=\sum_{i=1}^\infty (\lambda_i-i+\frac{1}{2})^k-(-i+\frac{1}{2})^k.
\end{equation}

We obtain a weight grading on $\Lambda^\ast$ by giving $p_i$ weight $i$.

Shifted symmetric functions play an important role in the study of quasimodular forms. The graded algebra of quasimodular forms is given by
$\widetilde{M}=\mathbb{Q}[P,Q,R]$ for 
\begin{equation}
    P=-24G_2,\quad Q=240G4,\quad R=-504 G_6
\end{equation}
in Ramanujan's notation for Eisenstein series

\begin{equation}
    G_k(q)=-\frac{B_k}{2k}+\sum_{r=1}^\infty\sum_{m=1}^\infty m^{k-1}q^mr,
\end{equation}
where $B_k$ denotes the $k$-th Bernoulli number. The grading on $\widetilde{M}$ is obtained by giving $G_k$ weight $k$.

Finally, given a function $f$ on partitions, we associate its $q$--bracket
\begin{equation}
    \langle f\rangle_q=\frac{\sum_\lambda f(\lambda)q^{|\lambda|}}{\sum_\lambda q^{|\lambda|}}\in\mathbb{C}[[q]].
\end{equation}
Then, the celebrated Bloch--Okounkov theorem \cite[Theorem 0.5]{bloch2000character} states that when $f$ is a shifted symmetric polynomial, its $q$--bracket is quasimodular. Moreover, it has weight $k$ as a quasimodular form.

\subsection{(Weighted) elliptic Hurwitz numbers}
\label{sec-weightell}
In this section, we introduce a new notion of weighted elliptic Hurwitz numbers. We again begin with the classical story.

\begin{definition}
\label{def-elliptichur}
    Let $g$ be a non--negative integer, $d$ a positive integer and $\mu^1,\dots,\mu^n$ partitions of $d$. Moreover, let $r=2g-2+nd-\sum_i\ell(\mu^i)$. Let $E$ be a genus $1$ Riemann surface and $p_1,\dots,p_n,q_1,\dots,q_r\in E$ pairwise distinct. For $\underline{\mu}=(\mu^1,\dots,\mu^n)$, we define an \textbf{elliptic Hurwitz cover} of type $(g,\underline{\mu})$ to be a holomorphic map $f\colon S\to E$ with
    \begin{itemize}
        \item $S$ is compact of genus $g$,
        \item $f$ has ramification profile $\mu^i$ over $p_i$,
        \item $f$ has ramification profile $(2,1,\dots,1)$ over $q_j$,
        \item $f$ is unramified else.
    \end{itemize}
    We call two elliptic Hurwitz covers $f_1\colon S_1\to E$ and $f_2\colon S_2\to E$ equivalent if there exists an isomorphism $g\colon S_1\to S_2$ with $f_1=f_2\circ g$.
    Then, we define the associated elliptic Hurwitz number as
    \begin{equation}
        N_{g,d}^\bullet(\underline{\mu})=\sum_{[f]}\frac{1}{|\mathrm{Aut}(f)|},
    \end{equation}
    where the sum is over all equivalence classes of elliptic Hurwitz covers of type $(g,\underline{\mu})$.

    Analogously to before, by imposing $S$ to be connected we obtain the connected numbers $N_{g,d}^\bullet(\mu^1,\dots,\mu^n)$.
\end{definition}

\begin{remark}
\label{rem-elliptic}
\begin{itemize}
    \item We may extend this definition for partitions $\mu^1,\dots,\mu^n$ not necessarily of the same number, by denoting with $N_{g,d}^\bullet(\mu^1,\dots,\mu^n)$ the Hurwitz numbers obtained for the partition $(\mu^i,1^{d-|\mu^i|}$, i.e. by adding $1$s until one reaches degree $d$. Observe that the number of simply ramified points $r$ does not change.
    \item When there are not partitions $\mu^i$, i.e. only simple ramification, we have $r=2g-2$ and we write
    \begin{equation}
        N_{g,d}^\bullet=N_{g,d}^\bullet(\emptyset).
    \end{equation}
\end{itemize}
\end{remark}

Again, there is an interpretation of $N_{g,d}^\bullet(\mu^1,\dots,\mu^n)$ as a factorisation problem in the symmetric group $S_d$. Indeed for the same set-up as in, we call $(\sigma_1,\dots,\sigma_n,\tau_1,\dots,\tau_r,\alpha,\beta)$ an elliptic factorisation of type $(g,\underline{\mu})$ if
\begin{itemize}
    \item $\sigma_i,\tau_j,\alpha,\beta\in S_d$,
    \item $C(\sigma_i)=\mu^i$, $C(\tau_j)=(2,1\dots,1)$,
    \item $\sigma_1\cdots\sigma_n\tau_1\cdots\tau_r=[\alpha,\beta]$, where $[\alpha,\beta]=\alpha\beta\alpha^{-1}\beta^{-1}$ is the commutator.
\end{itemize}

Then, we have

\begin{equation}
    N_{g,d}^\bullet(\underline{\mu})=\frac{1}{d!}|\{(\sigma_1,\dots,\sigma_n,\tau_1,\dots,\tau_r,\alpha,\beta)\,\textrm{of type}\,(g,\underline{\mu})\}|.
\end{equation}

We obtain the connected version when additional imposing that $\langle \sigma_i,\tau_j,\alpha,\beta\rangle$ is a transitive subgroup of $S_d$.

We now aim to define weighted elliptic Hurwitz numbers. We start with the following important notion.

\begin{definition}
    We denote by $M_{\underline{\mu}}^\lambda$ the number of elliptic factorisation $(\sigma_1,\dots,\sigma_n,\tau_1,\dots,\tau_r,\alpha,\beta)$ of type $(g,\underline{\mu})$, such that for $\tau_i=(r_i\, s_i)$ we have $s_i\le s_{i+1}$ and $(\tau_1,\dots,\tau_r)$ has signature $\lambda$.
\end{definition}

As in the double Hurwitz number setting, we consider two weight generating functions $G(z)$ and $\tilde{G}(z)$.

We define \textbf{weighted elliptic Hurwitz numbers} as
\begin{equation}
    N_{g,d}^{G,\bullet}(\underline{\mu})=\frac{r!}{d!}\sum_{\lambda\vdash r}G_\lambda M_{\underline{\mu}}^\lambda\quad\textrm{and}\quad N_{g,d}^{\tilde{G},\bullet}(\underline{\mu})=\frac{r!}{d!}\sum_{\lambda\vdash r}\tilde{G}_\lambda M_{\underline{\mu}}^\lambda.
\end{equation}

\begin{remark}
    We note that by the same arguments as in \cite{Guay_Paquet_2017}, we obtain classical, monotone and strictly monotone elliptic Hurwitz numbers for $G(z)=e^z$, $\tilde{G}(z)=\frac{1}{1-z}$ and $G(z)=1+z$ respectively.
\end{remark}

We now aim to related weighted elliptic Hurwitz numbers to shifted symmetric functions. For this, we first define the Jucys--Murphy elements

\begin{equation}
    J_n=\sum_{i=1}^{n-1} (i\,n)
\end{equation}
in the symmetric group algebra $\mathbb{C}[S_d]$. Recall that $\mathcal{Z}\mathbb{C}[S_d]$ the center of the symmetric group algebra is a vector space with the conjugacy classes
\begin{equation}
    C_\mu=\sum_{C(\sigma)=\mu}\sigma
\end{equation}
as a basis. It is a classical result \cite{JUCYS1974107,MURPHY1981287} that for a symmetric polynomial $f$, we have that $f(J_2,\dots,J_d)\in\mathcal{Z}\mathbb{C}[S_d]$.

Recall that $\mathfrak{K}=\sum_{\alpha,\beta\in S_d}[\alpha,\beta]$ the commutator sum is also an element of $\mathcal{Z}\mathbb{C}[S_d]$. We make the key observation that

\begin{equation}
    M_{\underline{\mu}}^{\lambda}=[\mathrm{id}]C_{\mu^1}\cdots C_{\mu^n}m_\lambda(J_2,\dots,J_d)\mathfrak{K}.
\end{equation}

It is now a standard argument (see e.g. \cite[Proposition 4.1]{Hahn_2022}) to show that

\begin{equation}
     M_{\underline{\mu}}^{\lambda}=d!\sum_{\gamma\vdash d}\prod_{i=1}^n f_{\mu^i}(\gamma)m_{\lambda}(\mathrm{cont}_\gamma),
\end{equation}
where $\mathrm{cont}_\gamma=(j-i)_{(i,j)\in \gamma}$ the usual content. It is well--known that $f_{\mu^i}(\gamma)$ is a shifted symmetric function in $\gamma$. Moreover, for a symmetric function $f$, we have that $f(\mathrm{cont}_\gamma)$ is a shifted symmetric function in $\gamma$ as well, see e.g. \cite{kerov1994polynomial}. 

Therefore, we obtain

\begin{equation}
\label{equ:shiftsymhur}
    N_{g,d}^{G,\bullet}(\underline{\mu})=r!\sum_{\lambda\vdash r}\sum_{\gamma\vdash d} G_\lambda\prod_{i=1}^n f_{\mu^i}(\gamma)m_{\lambda}(\mathrm{cont}_\gamma).
\end{equation}

\section{Bosonification is tropicalisation}
\label{sec-bostrop}
In this section, we derive a tropical interpretation for weighted Hurwitz numbers. We begin by bosonifying a fermionic expression for $H_g^{G,\bullet}$ and then we apply Wick's theorem, which allows to express vacuum expectations in the bosonic Fock space via tropical covers.

Weighted Hurwitz numbers indeed admit an interpretation as vacuum expectations in the fermionic Fock space, see e.g. \cite{Harnad_2015}.
We fix notation
\[
\hat{\gamma}_+(\bt) \defeq \exp\left(\sum_{i=1}^\infty p_i(\mathbf{t}) \frac{\alpha_i}{i} \right), \quad \hat{\gamma}_{-}(\bs) \defeq \exp \left( \sum_{i=1}^\infty p_i(\mathbf{s}) \frac{\alpha_{-i}}{i} \right).
\]

We note that in general (non--deformed) hypergeometric $\tau$--functions are obtained as
\begin{equation}
    \tau_\rho(\mathbf{t},\mathbf{s})=\langle 0\mid\hat{\gamma}_+(\bt)e^{\sum_{i\in\mathbb{Z}}T_iE_{i,i}}\hat{\gamma}_{-}(\bs)\mid0\rangle,
\end{equation}
with $r_i=e^{T_{i-1}-T_i}$ \cite{orlov2001hypergeometric}. In order to obtain weighted Hurwitz numbers, we need to introduce a $\beta$--deformation of the parameters $T_i$ as follows

\begin{equation}
    T_j^{(G,\beta)}=\sum_{i=1}^j\mathrm{log}(G(i\beta))\quad\textrm{and}\quad T_{-j}^{(G,\beta)}=-\sum_{i=0}^{j-1}\mathrm{log}(G(-i\beta)).
\end{equation}

From this, we define the fermionic operator

\begin{equation}
    \hat{C}_\rho=e^{\hat{A}},\quad \hat{A}=\sum_{i\in\mathbb{Z}}T_i^{(G,\beta)}E_{i,i}.
\end{equation}
We then have

\begin{equation}
    \tau^{(G,\beta)}(\mathbf{t},\mathbf{s})=\langle 0\mid\hat{\gamma}_+(\bt)e^{\sum_{i\in\mathbb{Z}}T_i^{(G,\beta)}E_{i,i}}\hat{\gamma}_{-}(\bs)\mid0\rangle
\end{equation}
and thus obtain

\begin{equation}
    H_g^{G,\bullet}(\mu,\nu)=[p_\mu(\mathbf{t})p_\nu(\mathbf{s})\beta^r]\langle 0\mid\hat{\gamma}_+(\bt)e^{\sum_{i\in\mathbb{Z}}T_i^{(G,\beta)}E_{i,i}}\hat{\gamma}_{-}(\bs)\mid0\rangle,
\end{equation}
where $r=2g-2+\ell(\mu)+\ell(\nu)$. Next, we define coefficients $A_i$ implicitly via

\begin{equation}
    \sum_{k=0}^\infty\beta^kA_kz^k=\mathrm{log}(G(\beta z)).
\end{equation}

In \cite[Section 5]{alexandrov2018fermionic}, an expansion of the operator $\hat{A}$ was derived. This expansion was obtained in exactly the operators $\mathcal{G}_{k+1}$ we defined \cref{equ-gviafoper} that were used in \cite{hahn2019tropicaljucyscovers} to study monotone and strictly monotone Hurwitz numbers. More precisely, it was proved that

\begin{equation}
    \hat{A}=\sum_{k=0}^\infty \beta^kA_k\mathcal{G}_{k+1}.
\end{equation}

\begin{remark}
    Essentially, we have that the coefficients $G_i$ define weighted Hurwitz numbers via Jucys--Murphy elements using monomial symmetric functions, while the coefficients $A_i$ describe the same numbers when using the symmetric power sums instead. While tropicalisation via the monomial basis is more straightforward (this was essentially the approach taken in \cite{hahn2019monodromy}), the power sum basis turns out to be more useful to investigate structural properties of the resulting vaccuum expectations from the tropical perspective.
\end{remark}

Putting everything together, we obtain the following lemma.

\begin{lemma}
    We have
    \begin{equation}
        H_g^{G,\bullet}(\mu,\nu)=\frac{1}{\prod_{i=1}^{\ell(\mu)} \mu_i \prod_{i=1}^{\ell(\nu)} \nu_i} \sum_{\lambda \vdash r} \frac{1}{\ell(\lambda)!} \Bigg<\prod_{i=1}^{\ell(\mu)}\alpha_{\mu_i} \prod_{j=1}^{\ell(\lambda)} A_{\lambda_j} \gb_{\lambda_j +1} \prod_{k=1}^{\ell(\nu)} \alpha_{-\nu_k} \Bigg>.
    \end{equation}
\end{lemma}

We then proceed to write this in a formulation so that our vacuum expectation consists solely of $\alpha$ operators.
\begin{proposition}\label{actuallybosonic}
\begin{align}
H_g^{G,\bullet}(\mu, \nu) =&  \frac{1}{\prod_{i=1}^{\ell(\mu)} \mu_i \prod_{i=1}^{\ell(\nu)} \nu_i} \sum_{\lambda \vdash r} \frac{1}{\ell(\lambda)!}  \\
& \times \prod_{j=1}^{\ell(\lambda)}\Bigg[ A_{\lambda_j}(\lambda_j )!  \sum_{\substack{g_1^j,g_2^j = 0 \\ \mathbf{x}^j \in S\Z^{l+2-2g_1^j-2g_2^j}}}  \Big< \tau_{2g_2^j -2}(\omega) \Big>_{g_2^j}^{\rsph, \circ} \Big< \mathbf{x}^{j,+} \mid \tau_{2g_1^j -2+ \ell(\mathbf{x}^{j,+}) + \ell(\mathbf{x}^{j,-})}(\omega) \mid \mathbf{x}^{j,-}   \Big>_{g_1^j}^{\rsph, \circ} \Bigg]\\
& \times \Bigg<\prod_{i=1}^{\ell(\mu)}\alpha_{\mu_i} \prod_{j=1}^{\ell(\lambda)} \Bigg(\prod_{0 > x_i^j \in \mathbf{x}^j} \alpha_{x_i^j}  \prod_{0 < x_k^j \in \mathbf{x}^j} \alpha_{x_k^j} \Bigg)\prod_{k=1}^{\ell(\nu)} \alpha_{-\nu_k} \Bigg>.
\end{align}
\end{proposition}

\begin{example}[Completed cycle Hurwitz numbers]
\label{ex-compltcycles}
   We come back to the important case of completed cycles Hurwitz numbers. As noted in \cite{Bychkov_2024}, and with appropriate modification for notational change, the weight generating function is
    \[
    G(z)=\exp\left(\mathcal{S}( \partial_z) \frac{z^r}{r!}\right).
    \]
    
    In the Fock space formalism the $(r+1)$ completed cycles Hurwitz numbers are defined as
    \[
    H^{(r)}_g(\mu, \nu) = \frac{1}{\prod\mu_i\prod\nu_j}\left<  \prod_{i=1}^{\ell(\mu)} \alpha_{\mu_i} \fb_{r+1}^s \prod_{j=1}^{\ell(\nu)} \alpha_{-\nu_j}\right>.
    \]

    This expression is purely in terms of the operators $\fb_r$ instead of the operators $\gb_r$. Essentially, the weight generating function determines a change of bases between these two families of operators. To see this, recall from \cref{equ-gviafoper} that $\mathcal{G}_{l+1}=l!\sum_{k\ge0}c_{2k-1}\mathcal{F}_{l-(2k-1)}$, where $c_{2k-1}$ are the coefficients of the expansion of $\varsigma(z)$. Thus, we have $\mathcal{G}_{l+1}=l!\left(\fb_{l+1}+c_1\fb_{l-1}+\dots\right)$. Since $\frac{\mathcal{S}(z)}{\varsigma(z)}=1/z$ the above choice of weight generating function causes all terms except $\fb_{l+1}$ to cancel.
\end{example}

\begin{definition}
    We define $\Gamma(\mathbb{P}^1_{trop}; r, \mu, \nu)$, the set of admissible tropical covers with at most $r$ vertices with ends determined by $\mu$ and $\nu$, both partitions of some positive integer $d$. This is the set of tropical covers $\pi : \Gamma \to \mathbb{P}^1_{trop}$, with the genus $g$ defined by the Riemann Hurwitz condition $2-2g = \ell(\mu) + \ell(\nu) - r$, where we fix $r$ points $p_1, \dots, p_r$ on the codomain $\mathbb{P}^1_{trop}$ such that:

\begin{enumerate}
\item The unbounded left pointing ends of $\Gamma$ have weights given by $\mu$, and the unbounded right pointing ends have weights given by $\nu$.
\item There exists some $l \leq r$, that is the number of vertices of $\Gamma$. We denote $V(\Gamma) = \{ v_1, \dots, v_l\}$ as the set of vertices. Thus, $\pi(v_i) = p_i$ and we let $w_i = \operatorname{val}(v_i)$ be the corresponding valences.
\item To each vertex $v_i$ of $\Gamma$ we have two associated integers $(g_1^i,g_2^i) \in \Z_{\geq 0}^2$, for $i = 1 , \dots, , l $, where we then require $g(v_i) = g_1^i + g_2^i$ as the genus at $v_i$, and with the following Euler characteristic requirement,
\[
\beta^1(\Gamma) + \sum_{i = 1}^l g(v_i) = g.
\]
\item We construct a partition $\lambda$ of length $l$, such that $\lambda_i = \operatorname{val}(v_i) + 2g(v_i) -2$, and require the Riemann Hurwitz condition;
\[
\sum_{i=1}^l \lambda_i = 2g-2 + \ell(\mu) + \ell(\nu).
\]
\end{enumerate}

We can further attach $G$-weighted vertex multiplicities $m_{G,v_i}$ for vertices in these graphs, let $\mathbf{x}^+$ be the right hand side weights, and $\mathbf{x}^-$ the left hand side weights, We define the multiplicity $m_{G, v_i}$ of $v_i$ as
 \[
 m_{G, v_i}  =  \lambda_i! | \hspace{0em}\aut(\mathbf{x}^+) \hspace{0em}| | \hspace{0em} \aut(\mathbf{x}^-) \hspace{0em}| A_{\lambda_i} \Big< \tau_{2g_2^i -2}(\omega)  \Big>_{g_2^i}^{\rsph, \circ} \Big< \mathbf{x}^{+}  \mid \tau_{2g_1^i -2+ \ell(\mathbf{x}^{+}) + \ell(\mathbf{x}^{-})}(\omega) \hspace{0em}\mid \hspace{0em} \mathbf{x}^{-}  \hspace{0em} \Big>_{g_1^i}^{\rsph, \circ} ,
 \]
 with the factor $A_{\lambda_i}$ coming from the expansion $\log(G(\beta x) )= \sum_{k=0}^{\infty}  A_k (\beta x)^{k}$.
\end{definition}


\begin{theorem}[Tropical weighted double Hurwitz numbers]\label{thm-tropcorr}
For a weight generating function $G(x)$ (or $\tilde{G}$) and two partitions of some fixed positive integer $ \mu, \nu \vdash d$, and $g$ a non--negative integer. Let $r=2g-2+\ell(\mu)+\ell(\nu)$

\[
H_g^{G,\circ}(\mu, \nu) = \sum_{\pi \in \Gamma(\mathbb{P}^1_{trop}; r, \mu, \nu)} \frac{1}{
|\aut (\pi)|} \frac{1}{\ell(\lambda)!} \prod_{v \in V(\Gamma)} m_{G,v} \prod_{e \in E^0(\Gamma)} \omega_e
\]
for $\Gamma(\mathbb{P}^1_{trop}; r, \mu, \nu)$ the set of admissible tropical covers $\pi : \Gamma \to \mathbb{P}^1_{trop}$, and $\lambda$ the corresponding partition for each given cover.
\end{theorem}

\begin{proof}
From any summand of our bosonic expression (c.f. \Cref{actuallybosonic}, where here we are solely looking at the vacuum expectation and not prefactors, we look at the connected case here but its exactly analogous so long as we consider connected vacuum expectations) over $\lambda_i$ we proceed. We have $g_1^i, g_2^i \in \Z_{\geq 0}$, and $\mathbf{x}^i \in S\Z^{\lambda_i + 2 -2g_1 - 2g_2}$ with $x_1^i \leq \cdots \leq x_{l_i}^i < 0 < x_{l_i +1}^i \leq \cdots \leq x_{\lambda_i + 2 - 2g_1^i -2g_2^i}$ for some $l_i \in [\lambda_i + 2 -2g_1^i - 2g_2^i]$ for all $i \in [\ell(\lambda)]$. With this, we apply Wick's Theorem in the formulation of \cite[Proposition 2.1]{hahn2019tropicaljucyscovers} to the vacuum expectation. This will give that 
\[
\Bigg<\prod_{i=1}^{\ell(\mu)}\alpha_{\mu_i} \prod_{i=1}^{\ell(\lambda)}  \alpha_{x_1^i} \dots \alpha_{x^i_{\lambda_i + 2 - 2g_1^i - 2g_2^i}}\prod_{k=1}^{\ell(\nu)} \alpha_{-\nu_k}\hspace{-.2em} \Bigg>^{\hspace{-.3em}\circ} \hspace{-.2em}= \hspace{-.5em}\sum_{\pi \in \Gamma(\mathbb{P}^1_{trop}, 0; \mu, \nu)} \hspace{-.3em}\frac{1}{|\aut(\pi)| }\prod_{i=1}^{\ell(\lambda)} | \aut (\mathbf{x}^{i,+})| | \aut (\mathbf{x}^{i,-})| \hspace{-.6em}\prod_{e \in E(\Gamma)} \hspace{-.5em} \omega_{e}.
\]
These covers will immediately satisfy (a) and (b). Note however that the genus condition we require will not be satisfied, all of the tropical covers here will be explicit. Thus, we need to attach further genera to our points to construct covers to satisfy the remainder of the conditions.

For this, consider some $\lambda \vdash r$, and some fixed data $g_1^i, g_2^i$ for $i =1, \dots ,\ell(\lambda)$. To resolve our issues, we set $g(v_i) = g_1^i + g_2^i$. This is clearly a bijection, and it also preserves automorphisms. We just need to check that this gives a genus we actually desire, that is, we verify $\beta^1(\Gamma) + \sum_{i=1}^l g(v_i) = g$. The Euler characteristic for a graph is $|V| - |E| = 1 - \beta^1(\Gamma)$ in terms of $|V|$ the number of vertices, and $|E|$ the number of edges. Note that $|V|  = \ell(\mu) + \ell(\nu) + \ell(\lambda)$. To compute $|E|$, the handshake lemma can be used to see
We then substitute $|\lambda| = 2g-2 + \ell(\mu) + \ell(\nu)$ for
\[
|E| = \frac 12  \left( 2g-2 + 2 
   \ell(\lambda) -2 \sum g(v_i) + 2 \ell(\mu) + 2 \ell(\nu)\right) = g-1 +  \ell(\mu) +  \ell(\nu) +  
   \ell(\lambda) - \sum g(v_i)  .
\]
Then substituting into the Euler characteristic requirement, we immediately get by cancellation that
\[
1- \beta^1(\Gamma) = 1 - g + \sum_{i=1}^l g(v_i).
\]
This yields $g = \beta^1 (\Gamma) + \sum_{i=1}^l (g_1^i + g_2^i)$ as required. These new tropical covers then satisfy (c) and (d) by construction.

All that remains to do is to consider the prefactors as global or local multiplicities. The factor that remains to be dealt with is 
\[
\frac{1}{\prod_{i=1}^{\ell(\mu)} \mu_i \prod_{i=1}^{\ell(\nu)} \nu_i}  \frac{1}{\ell(\lambda)!}    \prod_{j=1}^{\ell(\lambda)}\Bigg[ A_{\lambda_j}\lambda_j!  \Big< \tau_{2g_2^j -2}(\omega) \Big>_{g_2^j}^{\rsph, \circ} \Big< \mathbf{x}^{j,+} \hspace{0em}\mid \hspace{0em} \tau_{2g_1^j -2+ \ell(\mathbf{x}^{j,+}) + \ell(\mathbf{x}^{j,-})}(\omega)\hspace{0em} \mid \hspace{0em} \mathbf{x}^{j,-}   \hspace{0em}\Big>_{g_1^j}^{\rsph, \circ} \Bigg],
\]
while our current weighting is 
\[
\sum_{\pi \in \Gamma(\mathbb{P}^1_{trop};r, \mu, \nu)} \frac{1}{|\aut(\pi)| }\prod_{i=1}^{\ell(\lambda)} | \aut (\mathbf{x}^{i,+})| | \aut (\mathbf{x}^{i,-})| \prod_{e \in E(\Gamma)} \omega_{e}.
\]
With this, if we set the multiplicity of the $i$-th vertex as in requirement (e) and weight each cover by 
\[
\frac{1}{\prod\mu_i \prod \nu_j} \frac{1}{\ell(\lambda)!} \frac{1}{|\aut(\pi)|} \prod_{e \in E(\Gamma)} \omega_e \prod m_{G,v_i},
\]
we obtain the correct contribution. To conclude, note the weights of the unbounded edges of each cover contribute a factor of $\prod{\mu_i} \prod \nu_j$ in the product, if we take the product weights over all interior edges instead, this cancels to yields the desired
\[
\frac{1}{\ell(\lambda)!} \frac{1}{|\aut(\pi)|} \prod_{e \in E^0(\Gamma)} \omega_e \prod m_{G,v_i},
\]
 completing the proof.   
\end{proof}

\begin{example}
   We can recover the tropical Jucys covers cases (monotone and strictly monotone Hurwitz numbers) treated in \cite{hahn2019tropicaljucyscovers}. Here, the weight generating functions are $\frac{1}{1-z}$ and $1+z$ for the monotone and strictly monotone case respectively. For the strictly monotone case, the expansion $\log\left(\frac{1}{1-\beta x}\right) = \beta x + \frac{(\beta x)^2}{2} + \frac{(\beta x)^3}{3} + \cdots $ gives weights $A_i = \frac{1}{i}$. This coincides with the tropical interpretation in \textit{loc.cit.}, where the operators $\gb_{r+1}$ are instead defined as $(r-1)!\sum c_{2k-1} \fb_{r-(2k-1)}$, which would equal $A_{r} \gb_{r+1}$ in our notation with the given weighting function. The strictly monotone case has signs arising from the Taylor expansion $\log(1 + \beta x ) = \sum_{n=1}^\infty (-1)^{n+1} \frac{(\beta x)^n}{n} $, agreeing with the signs arising from valencies present there.
\end{example}

\section{Applications to polynomiality}
\label{sec-poly}
Building on previous work of Cavalieri, Johnson, and Markwig in \cite{cavalieri2010tropical,cavalieri2010chamberstructuredoublehurwitz}, we study the polynomiality of weighted double Hurwitz numbers. We will focus on the case for a weight generating function $G$, but note that the same arguments work for $\tilde{G}$.

As it will be more convenient to work with the number of simply ramified points $r$ instead of the genus, we introduce the following notation:

For any positive integer $n$ and $\mathbf x \in \Z^n$ such that $\sum_{i=0}^n x_i = 0$, we denote $H_G^r(\mathbf x)$ as the weighted double Hurwitz number $H_g^{G,\circ}(\mathbf{x}^+,\mathbf{x}^-)$ with $\mathbf{x}^+ = \{ x_i \mid x_i > 0\}$ and $\mathbf{x}^- = \{ x_i \mid x_i < 0\}$.

\subsection{Piecewise polynomiality}

\begin{definition}[Resonance Arrangement]\label{resonancearrangement}
The resonance arrangement is the hyperplane arrangement in $\R^n$ given by
\[
W_I = \left\{ \mathbf{x} \in \Z^n \mid \sum_{i \in I} x_i = 0\right\}
\]
for all $I \subsetneq \{ 1, \dots , n \}$. The connected components of the complement of the resonance arrangement are called chambers, which we denote $\mathfrak c$.
\end{definition}

\begin{theorem}[Piecewise Polynomiality for weighted double Hurwitz numbers]\label{thm-piecewisepoly}

Let $r$ be a non-negative integer, and take some fixed length $n$ of $\mathbf x$. Take some arbitrary weight generating function $G$ with a product decomposition. Then the function;
\[
H_G^r(\mathbf x) : \left\{\sum_{i=0}^n x_i = 0 \right\} \subseteq \Z^n \to \Q
\]
is a polynomial of degree $4g-3+n$ in $\mathbf{x}$ in each chamber of resonance $\mathfrak{c}$ of the complement of the resonance arrangement.
\end{theorem}

  The rest of the section is devoted to constructing appropriate formalisms, so that we can prove our piecewise polynomiality.

\subsubsection{Refining weighted double Hurwitz numbers}
We construct extra structures on weighted double Hurwitz numbers using our tropical correspondence as well as a suitable refinement.

\begin{definition}
Let $r$ be a non-negative integer, and $\mathbf{x} \in \Z^n$ such that $g$ is a non-negative integer. Take some ordered partition $\lambda'$ of $r$. We define
\[
H_{G, \lambda'}^r(\mathbf{x}) = \sum_{\pi \in \Gamma(\rsph_{trop}, \lambda;r, \mathbf{x}^+,  \mathbf{x}^-)} \frac{1}{
|\aut (\pi)|} \frac{1}{\ell(\lambda')!} \prod_{v \in V(\Gamma)} m_{G,v} \prod_{e \in E(\Gamma)} \omega_e
\]
over the set of tropical covers as in \Cref{thm-tropcorr}, but with the additional condition that we only admit covers that yield the partition $\lambda'$.

Further, for an unordered partition $\lambda''$, we define
\[
H_{G, \lambda''}^r(\mathbf{x}) = \sum_{\lambda'} H_{G, \lambda'}^r(\mathbf{x})
\]
summing over ordered partitions $\lambda'$ obtained by some ordering of $\lambda ''$.
\end{definition}

\begin{remark}\label{indexpartition}
By construction, we note that
\[
H_{G}^r(\mathbf{x}) = \sum_{\lambda'}  H_{G, \lambda'}^r(\mathbf{x}) = \sum_{\lambda ''} H_{G, \lambda''}^r(\mathbf{x})
\]
for the first sum over all ordered partitions $\lambda'$ of $r$, and the second taken over all unordered partitions $\lambda ''$ of $r$.
\end{remark}

We must construct hyperplanes which cut out our chambers where polynomiality will hold.
We work with certain combinatorial formalism to prove our results, which we must also define. We need an alteration of our notions of tropical graphs.


\begin{definition}[Combinatorial cover]
For fixed $g$, $\mathbf{x} = (x_1, \dots, x_n) \in \Z^n$, $\lambda \vdash r = 2g-2+n$ unordered, a graph $\Gamma$ is a combinatorial cover of type $(g, \mathbf{x}, \lambda)$ if;

\begin{enumerate}[label=(\alph*)]
    \item $\Gamma$ is a connected graph with at most $r$ vertices; 
    \item $\Gamma$ has $n$ 1-valent vertices (leaves), with adjacent ends labelled by the weights $x_1, \dots, x_n$ all oriented inwards. If $x_i > 0,$ we say it is an in-end, otherwise it is an out-end;
    \item The set of edges which are not ends are denoted $E^{0}(\Gamma)$;
    \item There are $\ell(\lambda)$ inner vertices;
    \item We denote the inner vertices $v_1, \dots, v_{\ell(\lambda)}$ and assign a non-negative genus of $v_i$, $g(v_i)$ with the condition that $\lambda_i = \operatorname{val}(v_i) + 2g(v_i) -2$;
    \item reversing the orientation of out-ends, $\Gamma$ does not have any sinks or sources;
    \item The internal vertices are ordered compatibly with the partial ordering induced by the directions of edges;
    \item $g = \beta_1(\Gamma) + \sum g(v_i)$ for $\beta_1(\Gamma)$ the first Betti number of $\Gamma$;
    \item every internal edge $e$ is equipped with a weight $\omega(e) \in \N$. The weights satisfy the balancing condition at every inner vertex.
\end{enumerate}

We denote $\Gamma(\mathbf{x}, \lambda, \mathfrak{d}, \mathfrak{o})$ for a graph with directed edges $(\mathfrak{d})$, and compatible vertex ordering $(\mathfrak{o})$.
\end{definition}

It is clear by construction that these combinatorial covers are simply our tropical covers, and using \Cref{indexpartition}, we have that
\begin{equation}\label{combinatorialcoverhur}
H_G^r(\mathbf{x}) = \sum_{\lambda \vdash r} \sum_\Gamma \frac{1}{|\aut(\Gamma)|} \frac{1}{\ell(\lambda)!} \varphi_\Gamma
\end{equation}
for the second summation indexing over all combinatorial covers $\Gamma$ of type $(g,\mathbf{x}, \lambda)$, where we have
\[
\varphi_\Gamma = \prod_{i = 1}^{\ell (\lambda)} m_{v_i} \prod_{e \in E^{0}(\Gamma)} \omega(e).
\]

\subsubsection{Hyperplane Arrangements}
We can view any $\mathbf{x}$-graph $\Gamma$ as a one-dimensional cell complex once we equip it with a fixed reference orientation. Then the differential $d: \R E_\Gamma \to \R V_\Gamma$ sends directed edges to the different of its head and tail vertices gives us a short exact sequence
\[
 0 \to \ker(d) \to \R E_\Gamma \to \im (d) \to 0.
\]
We then decompose $\R E_\Gamma = \R^n \bigoplus \R^{|E^0(\Gamma)|}$ into ends and internal edges. Then we will have the vector $(\mathbf{x}, 0) \in \im (d)$ when $\sum x_i = 0$.

\begin{definition}
The space of flows is constructed as
\[
F_{\Gamma}(\mathbf{x}) = d^{-1}(\mathbf{x}, 0).
\]
\end{definition}

Within this space, we have the hyperplane arrangement $\mathcal{A}_\Gamma(\mathbf{x})$ given by the restriction of the co-ordinate hyperplanes corresponding to internal edges, $E^0(\Gamma)$. We have a defining polynomial for the arrangement
\[
\varphi_{\mathcal{A}} = \prod e_i
\]
in terms of $e_i$, the co-ordinate functions on $\R E_\Gamma$ restricted to the space of flows $F_\Gamma(\mathbf{x})$.

We note that we prescribe some reference orientation on a given $\mathbf{x}$-graph. This corresponds to fixing some bounded chamber in the hyperplane arrangement.

\begin{lemma}[\cite{cavalieri2010chamberstructuredoublehurwitz}]\label{boundedchambers}

The bounded chambers of $\mathcal{A}_\Gamma$ correspond to orientations of $\Gamma$ which have no directed cycles. Further, for any $(\mathbf{x}, \lambda)$-graph $\Gamma$, the bounded chambers of $\ba_\Gamma$ are in bijection with directed $(\mathbf{x}, \lambda)$-graphs projecting to $\Gamma$ upon forgetting of edge orientations coming from a combinatorial cover.
\end{lemma}

We now want to rephrase our formulation in terms of combinatorial covers for our polynomiality to arise. 

\begin{definition}
We denote $S_\Gamma(\mathbf{x})$ the contribution to $H_{G,\lambda}^r(\mathbf{x})$ of all monodromy graphs with underlying $\mathbf{x}$-graph $\Gamma$, where the valences and degree at each vertex respect $\lambda_i = \operatorname{val}(v_i) +2g(v_i) -2$.
\end{definition}

\begin{definition}
For an $\mathbf{x}$-graph $\Gamma$, we refer to the chambers of $\ba_\Gamma(\mathbf{x})$ in the flow space $F_\Gamma(\mathbf x)$ as $F$-chambers.
For an $F$-chamber $A$, take $\Gamma_A$ as the directed $\mathbf x$-graph $\Gamma$ with edge directions corresponding to $A$. We denote $m(A)$ or $m(\Gamma_A)$ for the number of orderings of vertices of $\Gamma_A$. By \Cref{boundedchambers}, $m(A) = 0$ if and only if the chamber $A$ is unbounded. We use the notation $\operatorname{Ch}(\ba_\Gamma(\mathbf x))$ for the set of chambers of $\ba_\Gamma(\mathbf x)$, and $\mathcal{BC}_\Gamma(\mathbf x)$ denotes the set of bounded chambers.

The sign on $\varphi_{\ba}$ alternates on adjacent $F$-chambers, when we swap the direction of a single edge. We use $\operatorname{sign}(A)$ to denote the sign of $\varphi_\ba$ on $A$.

\end{definition}

Before we proceed any further, we first note that the vertex multiplicities are all polynomial terms. Recalling our vertex multiplicity of \Cref{thm-tropcorr}, then noting the modification that here we no longer preserve the information of the specific $g_1, g_2$ gives that
\begin{equation}\label{importantvertmult}
m_{G,v_i} \hspace{-.33em} = \hspace{-.2em} \lambda_i! | \hspace{-.2em}\aut(\mathbf{x}^+) \hspace{-.1em}| | \hspace{-0em} \aut(\mathbf{x}^-) | A_{\lambda_i} \hspace{-.4em}\sum_{g_1^i + g_2^i = g(v_i)} \hspace{-.6em}\Big< \hspace{-.2em}\tau_{2g_2^i -2}(\omega) \hspace{-.2em} \Big>_{g_2^i}^{\rsph, \circ} \hspace{-.2em} \Big< \mathbf{x}^{+} \hspace{-.2em} \mid \hspace{-.2em}\tau_{2g_1^i -2+ \ell(\mathbf{x}^{+}) + \ell(\mathbf{x}^{-})}(\omega) \hspace{-.2em}\mid \hspace{-.2em} \mathbf{x}^{-}  \hspace{-.2em} \Big>_{g_1^j}^{\rsph, \circ}\hspace{-.2em} ,
\end{equation}
we proceed.

Once we apply \Cref{polynomialgw}, we obtain that
\[
m_{G,v_i} = \lambda_i! A_{\lambda_i} \sum_{g_1^i + g_2^i = g(v_i)} \hspace{-1em}\Big< \hspace{-.2em}\tau_{2g_2^i -2}(\omega) \hspace{-.2em} \Big>_{g_2^i}^{\rsph, \circ} \hspace{-.2em} \left[w^{g_1^i}\right]\frac{\prod_j\sba( \mathbf{x}^{+}_jw) \prod_j \sba(\mathbf{x}^{-}_jw)}{\sba(w)}.
\]
We then note the Taylor expansions, $\mathcal{S}(w) = 1 + \frac{w^2}{24} + \frac{w^4}{1920} + O(w^6)$ and $\frac{1}{\mathcal S (w)} =1 - \frac{w^2}{24} + \frac{7w^4}{5760} + O(w^6) $ are both
power series. Thus, $m_{G,v_i}$ is a polynomial in the adjacent edge weights. We use the notation $M(v_i)$ for this polynomial (upon fixing our weight generating function $G$, and note that it is independent of the flow of the respective branching graph.)

Now we return to the space of flows. For all integer values of $\mathbf{x}$, $F_\Gamma(\mathbf x)$ will have an affine lattice, coming from the integral structure of $\Z E_\Gamma$. We denote this lattice as
\[
\Lambda = F_\Gamma (\mathbf x) \cap \Z E_\Gamma.
\]
This allows for an appetising interpretation of $S_\Gamma(\mathbf x)$ via the hyperplane arrangement $\ba_\Gamma(\mathbf x)$. Choosing edge weights, that is a choice of flow on $\Gamma$, corresponds to picking points in $\Lambda$. Therefore
\begin{align}
S_\Gamma(\mathbf x)& = \frac{1}{|\aut (\Gamma)|} \sum _{A \in \operatorname{Ch}(\ba_\Gamma(\mathbf x))} m(A) \sum_{f \in A \cap \Gamma} \left( \prod_{e \in E^0(\Gamma)} \omega(e) \prod_i M(v_i) \right) \\
&= \frac{1}{|\aut (\Gamma)|} \sum _{A \in \operatorname{Ch}(\ba_\Gamma(\mathbf x))} m(A) \operatorname{sign}(A)\sum_{f \in A \cap \Gamma} \left(  \varphi_{\ba}(f) \prod_i M(v_i) \right).
\end{align}
Passing from the first to the second line arises once we use that the product of edge weights of the flow $f$ is the absolute value of $\varphi_\ba$ computed at $f = (e_i)_i$, so that if $f \in A$, this is just $\operatorname{sign}(A) \varphi_\ba(f).$


An essential idea underlying our proof here comes from Ehrhart theory.

\begin{theorem}[Weighted Ehrhart reciprocity \cite{ardila2015doublegromovwitteninvariantshirzebruch}]\label{ehrhartmagic}
Let $P$ be a lattice polytope in $\R^m$ and $f:\R^m \to \R$ a polynomial function. Then for each positive integer $n$, the functions
\[
L_{P,f}(n) = \sum_{z \in nP \cap \Z^m}f(z), \quad 
L_{P^\circ,f}(n) = \sum_{z \in nP^\circ \cap \Z^m}f(z),
\]
are polynomial.
\end{theorem}

\begin{remark}\label{Ehrhart continuation} 
The polynomials $L_{P,f}(n), L_{P^\circ,f}(n)$ for $f$ a degree $p$ polynomial, and the polytope $P$ of dimension $d$ will be a polynomial of degree at most $d+p$. This is a classically known result, but see for example \cite{deloera2024ehrhartfunctionsweightedlattice} for a proof of this fact.
\end{remark}

We first prove the statement on our refinement of weighted double Hurwitz numbers. It will then be straightforward to extend this to the standard enumerations. Our refinement means a wall crossing formula will be possible to derive. 

\begin{theorem}

Let $r$ be a non-negative integer, and take some fixed length $n$ of $\mathbf x$. Take some ordered partition $\lambda \vdash d$ and an arbitrary weight generating function $G$. Then the function
\[
H_{G, \lambda}^r(\mathbf x) 
\]
is piecewise polynomial of degree at most $4g-3+n$ in $\mathbf{x}$.
\end{theorem}

\begin{proof}
    Fix some $\mathbf x$-graph $\Gamma$ with reference orientation given by the flow $f$. Note that 
    \[
    \frac{1}{|\aut (\Gamma)|} \left( \varphi_\ba (f) \prod_i M(v_i) \right)
    \]
    is a polynomial of degree $|E^0(\Gamma)| + 2 \sum g(v_i)$ given $\varphi_\ba(f)$ is  a polynomial of degree $|E(\Gamma)|$, and each $M(v_i)$ is polynomial of degree $2g(v_i)$. Further, once we consider the Euler characteristic of our graph, we see that 
    \[
   | E(\Gamma)| = \beta_1(\Gamma) -1 + |V| = \beta_1(\Gamma) - 1 +n + \ell(\lambda).
    \]
    However we care solely about our inner edges, and we know we will have $n$ ends, so we see that 
    \[
    |E^0(\Gamma) |= \beta_1(\Gamma) - 1 + \ell(\lambda) = \ell(\lambda) - 1 + g - \sum g(v_i).
    \]
    We then compute the total degree of this polynomial to be
    \[
    |E^0(\Gamma)| + 2 \sum g(v_i) =  \ell(\lambda) - 1 + g +\sum g(v_i).
    \]

    We then seek an upper bound on this quantity. Note that $\lambda_i = \operatorname{val}(v_i) +2g(v_i) -2$ and $\operatorname{val}(v_i) \geq2$ so the quantity is maximised at $\lambda = (1,1, \dots, 1)$ so we get that
    \[
     |E^0(\Gamma)| + 2 \sum g(v_i)  \leq 3g-3+n.
    \]

    Thus at each point of the flow we must sum up. By \cite[Remark 2.11]{cavalieri2010chamberstructuredoublehurwitz}, each $F_\Gamma(\mathbf{x})$ is a $g$-dimensional affine space. 
    
   Then, by \Cref{ehrhartmagic} and \Cref{Ehrhart continuation}, we can consider $S_{\Gamma}(\mathbf x)$, where the chamber will be some polytope (or unbounded and thus zero) where we sum over the integral points, and determine it to be a polynomial of a specific degree, so long as the topology of $\ba_\Gamma(\mathbf x)$ does not change. It is important that each vertex is indeed an integer as otherwise we would only have quasipolynomiality. This results from the fact that the incidence matrix of a directed graph is totally unimodular. 

   The degree of the polynomial we already know to be $\ell(\lambda) + g + \sum g(v_i) -1$ and we have already stated the dimension of our polytope will be $\beta_1(\Gamma)$. Then, as each contribution to $H_{G, \lambda}^r$ is locally polynomial and we have a finite number of contributions, our total $H_{G, \lambda}^r(\mathbf{x})$ is locally polynomial of degree $\leq 4g -3 +n$ as desired.

\end{proof}

\subsubsection{Wall Structure}

With this, we have piecewise polynomiality for $H_{G, \lambda}^r(\mathbf{x})$, now we want to determine our chamber structure. As we mentioned, the polynomiality will hold wherever the topology of $\ba_\Gamma(\mathbf x)$ does not change upon varying $\mathbf x$. When translating generic hyperplane arrangements, we observe changes in topology whenever we pass through a non-transversality. However with our constraints, there can be non-transversalities which appear for every value of $\mathbf x$. The topology will only change upon passing through other nontransversalities. 
A nontransversality which occurs for every $\mathbf x$ is called good. We classify them as follows.

\begin{definition}
Consider a set $I$ of $k$ hyperplanes in $\ba_\Gamma(\mathbf x)$ (or equivalently edges in $\Gamma$) which intersect in codimension $k-l$. This intersection is good if there exists a set $L$ of $l$ vertices in $\Gamma$ so that $I$ is exactly the set of edges incident to vertices in $L$.
\end{definition}

\begin{definition}
The discriminant locus $\mathcal{D} \sub \R^n$ is the set of all values $\mathbf x$ such that there exists a directed $\mathbf x$-graph $\Gamma$ with its hyperplane arrangement $\ba_\Gamma(\mathbf x)$ having a nontransverse intersection that is not good.

The discriminant is a union of hyperplanes which we call the discriminant arrangement. We call the hyperplanes walls, and the chambers defined by this arrangement $H$-chambers.
\end{definition}

These $H$-chambers are the chambers of polynomiality for $H_{G,\lambda}^r(\mathbf x)$. We now need to verify that these chambers correspond to the resonance arrangements of \Cref{resonancearrangement}:
\[
\sum_{i \in I} x_i = 0 \vspace{-.5em},
\]
for $I \subsetneq \{ 1, \dots, n\}$.

\begin{definition}
A simple cut of a graph $\Gamma$ is a minimal set $C$ of edges which disconnected the ends of $\Gamma$. That is, for any two ends of $\Gamma$, any path between them contains an edge in $C$, and no proper subset of $C$ has this property.

For any $\mathbf x$-graph $\Gamma$, a flow in $F_\Gamma(\mathbf x)$ is said to be disconnected if there exists some simple cut $C$ with the flow on each edge of $C$ being zero.
\end{definition}

With this, if a flow is disconnected by the balancing condition means that the sum $\sum_{i \in I}x_i $ of weights of ends on any connected component of $\Gamma \backslash C$ is 0.

\begin{lemma}[\cite{cavalieri2010chamberstructuredoublehurwitz}]
The discriminant arrangement $\mathcal{D}$ is the set of $x \in \R^n$ such that there is some $\mathbf{x}$-graph $\Gamma$ where $F_\Gamma(\mathbf{x})$ admits a disconnected flow.
\end{lemma}

\begin{remark}
The proof presented in \cite{cavalieri2010chamberstructuredoublehurwitz} uses that the graph only has trivalent vertices, however the generalisation to higher valence is clear.
\end{remark}

\begin{theorem}
The walls for the discriminant arrangement are given by the resonances:
\[
\sum_{i \in I}x_i = 0, \quad \text{for any proper subset }I \subsetneq \{1, \dots, n\}.
\]
\end{theorem}

\begin{proof}
    We know that for every graph $\Gamma$, the walls of polynomiality of $S_\Gamma(\mathbf x)$ are the points $\mathbf x$ so that $\Gamma$ admits a disconnected flow. We know this is a subset of the resonance arrangement. For the reverse inclusion, take any resonance. It is now possible to construct a graph $\Gamma$ with some edge $e$, so that $\Gamma \backslash \{e\}$ has two components, one that has all the ends of $I$, and the other containing the edges of $I^c$.
\end{proof}

Thus this demonstrates that each $H_{G,\lambda}^r(\mathbf x)$ is piecewise polynomial in the resonance chambers. This further implies that the (finite) sum over all $\lambda \vdash r$ will also be piecewise polynomial, $H_G^r(\mathbf x)$. This then also proves \Cref{thm-piecewisepoly}.

\subsection{Wall-crossing formulae}
Here, we extend our piecewise polynomiality to a wall-crossing formula for weighted double Hurwitz numbers. Given our refinement, the result is immediate upon application of results from \cite{hahn2019wallcrossingrecursionformulaetropical}. There, the results are specifically for monotone and strictly monotone Hurwitz numbers but the techniques allow for vertex multiplicities and genera as we require for our graphs.

\begin{definition}
    Let $\mathfrak{c}_1, \mathfrak c_2$ be two $H$-chambers, adjacent along some wall $W_I$, taking $\mathfrak c_1$ as the chamber with $\mathbf x_I = \sum_{i \in I} x_i <0$. Let $P_{G,i}^\lambda(\mathbf{x})$ be the polynomial for $H_{G,\lambda}^r$ in $\mathfrak c_i$. Then  we define the wall crossing function as
    \[
    WC_{G, I}^\lambda (\mathbf{x}) = P_{G,2}^\lambda(\mathbf{x}) -P_{G,1}^\lambda(\mathbf{x}).
    \]
\end{definition}

\begin{theorem}\label{wallcrossingthm}
    Take $g$ a non-negative integer, a fixed length $n$ of $\mathbf{x}$ and $\lambda$ an unordered partition of $r = 2g-2+n$. Then we have
    \[
   WC_{G, I}^\lambda (\mathbf{x}) = \sum_{|\mathbf y| = |\mathbf z| = |\mathbf{x}_I|} \sum_{\substack{\lambda^i \text{ unordered} \\ \lambda^1 \cup \lambda^2 \cup \lambda^3 = \lambda}} \left( (-1)^{\ell(\lambda^2)} \frac{\prod y_i}{\ell(\mathbf{y})!} \frac{\prod z_i}{\ell(\mathbf{z})!} H^{|\lambda^1|}_{G, \lambda^1}(\mathbf{x}_I,-\mathbf y) H_{G, \lambda^2}^{|\lambda^2|, \bullet} (\mathbf{y},- \mathbf z) H_{G, \lambda^3}^{|\lambda^3|}(\mathbf z, \mathbf{x}_{I^C}) \right)
    \]
    for $\mathbf y$ an ordered tuple of length $\ell(\mathbf y)$ of positive integers with sum $\mathbf{z}$ (analogous for $\mathbf z$). The genus of graphs counted in $H^{|\lambda^1|}_{G, \lambda^1}(\mathbf{x}_I,\mathbf y)$ is given by $|\lambda^1| = 2g_1 -2 + \ell(\mathbf{x}_I) + \ell(\mathbf y)$ and similarly $g_2, g_3$ are determined.
\end{theorem}

\begin{definition}
    Take $\Gamma$ a directed graph, and $E$ a subset of the edges of $\Gamma$. We also consider the graph whose edges are the connected components of $\Gamma \backslash E^C$, with vertex set $E^C$. This graph is denoted the contraction of $\Gamma$ with respect ot $E$, we denote it by $\Gamma / E$.

    Fix a directed $\mathbf x$-graph $\Gamma_A$ and some subset $I \sub \{ 1, \dots , n \}$. Then the set $\operatorname{Cuts}_I(\Gamma_A)$ of $I$-cuts of $\Gamma_A$ is the set of all subsets $C$ of $E(\Gamma_A)$, such that $C = \varnothing$ or 
    \begin{enumerate}
        \item $\Gamma_A \backslash C$ is disconnected;
        \item The ends of $\Gamma_A$ lie on exactly two components of $\Gamma_A \backslash C$, one contains all ends indexed by $I$, the other all ends indexed by $I^C$;
        \item the directed graph $\Gamma / C^C$ is acyclic and has the connected component containing $I$ as the initial vertex, and the component containing $I^C$ as the final vertex.
    \end{enumerate}

    Take $v(\Gamma_A \backslash C)$ as the number of components of $\Gamma_A \backslash C$, to define the rank of $C$ as
    \[
    \operatorname{rk}(C) = v(\Gamma_A \backslash C) - 1.
    \]
\end{definition}
Then using a result of \cite{cavalieri2010chamberstructuredoublehurwitz}, we have
\[
WC(x_2) = \sum_\Gamma \sum_{A \in \mathcal{BC}_\Gamma (x_2)} \sum_{C \in \operatorname{Cut}_I(\Gamma_A)} (-1)^{\operatorname{rk}(C)-1} \binom{\ell(\lambda)}{s,t_1, \dots t_N, u} \left( \sum_{\Lambda \cap A}  \varphi_{\mathcal{A}} \prod_i M(v_i)\right)
\]
for $N = \operatorname{rk}(C)-1$ and $t_1, \dots t_N$ are the number of inner vertices of the $N$ inner components of $\Gamma_A \backslash C$.

\begin{definition}
    Let $\Gamma$ be an $\mathbf{x}$-graph, and $C \in \operatorname{Cuts}_I(\Gamma_A)$. Then $C$ is a thin cut if all edges in $C$ are adjacent to either the initial component containing $I$ or the final component containing $I^C$. For such a thin cut $T$, we construct $P(T)$, the set of all cuts $C \in \operatorname{Cuts}_I(\Gamma_A)$ which contain $T$.
\end{definition}

Then using the correct sign convention of \cite{hahn2019wallcrossingrecursionformulaetropical} we have
\[
(-1)^t \binom{\ell(\lambda)}{s,t,u} = \sum_{C \in P(T)} (-1)^{\operatorname{rk}(C)-1} \binom{\ell(\lambda)}{s,t_1, \dots t_N, u}, 
\]
which combines for us to say
\[
WC(x_2) = \sum_\Gamma \sum_{A \in \mathcal{BC}_\Gamma(x_2)} \sum_{\substack{ T \in \operatorname{Cut}_I(\Gamma_A) \\ \text{thin}}} (-1)^t \binom{\ell(\lambda)}{s,t,u} \left( \sum_{\Lambda \cap A}  \varphi_{\mathcal{A}} \prod_i M(v_i)\right).
\]

We note that each thin cut divides $\Gamma_A$ into three parts; the initial component $\Gamma_A^1$, an intermediate part $\Gamma_A^2$ (which may be disconnected) and a final component $\Gamma_A^3$. Note that $\Gamma_A^1$ contributes to $H^{|\lambda^1|}_{G, \lambda^1}(\mathbf{x}_I,-\mathbf y)$, $\Gamma_A^2$ to $ H_{G, \lambda^2}^{|\lambda^2|, \bullet} (\mathbf{y},- \mathbf z)$ and $\Gamma_A^3$ to $ H_{G, \lambda^3}^{|\lambda_3|}(\mathbf z, \mathbf{x}_{I^C})$, $\lambda^1 \cup \lambda^2 \cup \lambda^3 = \lambda$ and $\mathbf{y}, \mathbf{z}$ are some partitions such that $|\mathbf{y}| = | \mathbf{x}_I|, \ |\mathbf{z}| = |\mathbf{x}_{I^c}|$. Noting that 
\[
\varphi_{\Gamma_A} = \frac{\ell(\lambda_1)! \ell(\lambda_2)! \ell(\lambda_3)!}{\ell(\lambda)!} \frac{\prod y_i}{\ell(\mathbf{y})!} \frac{\prod z_i}{\ell(\mathbf{z})!}  \varphi_{\Gamma_A^1} \varphi_{\Gamma_A^2}\varphi_{\Gamma_A^3,}  
\]
and
\[
\binom{\ell(\lambda)}{s,t,u} = \binom{\ell(\lambda)}{\ell(\lambda_1), \ell(\lambda_2), \ell(\lambda_3)} = \frac{\ell(\lambda)!}{\ell(\lambda_1)! \ell(\lambda_2)! \ell(\lambda_3)!}
\]
combines to cancel a factor. This then completes the proof of \cref{wallcrossingthm}.

\section{Tropical mirror symmetry for weighted Hurwitz numbers}
\label{sec-tropmirr}
In this section, we study elliptic weighted Hurwitz numbers towards tropical mirror symmetry. The following is our first main result.

\begin{theorem}
\label{thm-quasimod}
For fixed $g \geq 2$, and $\underline{\mu}=(\mu^1,\dots,\mu^n)$ a tuple of partitions of arbitrary size. Then, we have that 
\[
\sum_d N_{g,d}^{G,\circ}(\underline{\mu}) q^d
\]
is a quasimodular form of mixed weight $\le 6g-6+\sum (4\ell(\mu^i-2|\mu^i|))$.
\end{theorem}

\begin{remark}
    Recall from \cref{rem-elliptic} that we add $1$s to the partitions in $\underline{\mu}$ when increasing $d$.
\end{remark}

\begin{proof}
    The result follows the same arguments as in \cite[Section 4]{Hahn_2022}, since we have already observed in \cref{equ:shiftsymhur} that $N_{g,d}^{G,\circ}(\underline{\mu})$ is an expression in terms of shifted symmetric functions.
\end{proof}

\subsection{Refined Quasimodularity}
Having shown that elliptic weighted Hurwitz numbers provide quasimodular forms, we now aim to refine this statement using tropical combinatorics. Tropical refinements of such quasimodularity statements have previously appeared for classical Hurwitz numbers in\cite{B_hm_2015,tropmirror2022,goujard2017countingfeynmanlikegraphsquasimodularity} and for monotone/strictly monotone Hurwitz numbers in \cite{Hahn_2022}. We focus on the special case $N_{g,d}^{G,\circ}$ with only simple ramification.

To begin with, we need to derive a tropical correspondence theorem for weighted elliptic Hurwitz numbers. We begin with the following definition of weighted tropical elliptic covers.

\begin{definition}
    Fix a weight generating  function $G$, a non-negative integer $g \geq 0$, an orientation on $E_{trop}$ and points $p_0,p_1, \dots, p_{2g-2}$ of $E_{trop}$ such that $p_0, p_1, \dots, p_{2g-2}$ are ordered according to the orientation. Let $\pi : \Gamma \to E_{trop}$ be a tropical cover of genus $g$ and degree $d$ such that $\pi^{-1}(p_0)$ contains no vertices, and $\Gamma$ has at most $2g-2$ vertices $v_1, \dots, v_l, \ l \leq 2g-2$. We force $v_i \in \pi^{-1}(p_i)$ for $i = 1, \dots, l$. We require $\lambda_i = \operatorname{val}(v_i) +2g(v_i) -2$ and obtain a composition $\lambda(\pi) = (\lambda_1, \dots , \lambda_l)$. If $|\lambda| = 2g-2$ we call $\pi$ an admissible elliptic tropical cover, of type $(g,d)$. We denote $\Gamma(E_{trop}; g, d)$ for the set of all admissible elliptic tropical covers of type $(g,d)$.

We associate a multiplicity to each cover $\pi \in \Gamma(E_{trop}; g, d)$, 
\[
\operatorname{mult}_G(\pi) = \frac{1}{
|\aut (\pi)|} \frac{1}{\ell(\lambda(\pi))!} \prod_{v \in V(\Gamma)} m_{G,v} \prod_{e \in E(\Gamma)} \omega_e
\]
with the vertex multiplicity precisely as in \Cref{thm-tropcorr} at each vertex (with $\mathbf{x}^+$ being the right hand weights, and $\mathbf{x}^-$ the left hand weights with respect to the orientation of $E_{trop}$).
\end{definition}

We then obtain the following tropical correspondence theorem for weighted elliptic Hurwitz numbers.

\begin{theorem}\label{thm-ellcorr}
Fix $g \geq 2, d >0$. Then;
\[
N_{g,d}^{G,\circ}  = \sum_{\pi \in \Gamma(E_{trop};g,d)} \operatorname{mult}_G(\pi).
\]
\end{theorem}

\begin{proof}
    The special case $G(z)=\frac{1}{1-z}$ was proved in \cite[Theorem 5.3]{Hahn_2022}. The same argument works here. The main idea is to cut up the tropical cover over the fibre $p_0$ to obtain a tropical cover contributing to $H_g^{G,\bullet}(\mu,\mu)$. On the factorisation side, this corresponds to rewriting $(\tau_1,\dots,\tau_r,\alpha,\beta)$  with
    \begin{equation}
        \tau_r\cdots\tau_1=[\alpha,\beta]
    \end{equation}
    as
    \begin{equation}
        \tau_r\cdots\tau_1\beta=\alpha\beta\alpha^{-1}.
    \end{equation}
    Thus, we may instead count tuples $(\tau_1,\dots,\tau_r,\sigma_1,\sigma_2,\alpha)$ with 
    \begin{equation}
        \tau_r\cdots\tau_1\sigma=\sigma_2\quad\textrm{with}\quad \alpha\sigma_2\alpha^{-1}=\sigma_1.
    \end{equation}

    But the subtuple $(\tau_1,\dots,\tau_r,\sigma_1,\sigma_2)$ contributes to the weighted double Hurwitz numbers $H_g^{G,\bullet}(\mu,\mu)$ as well. What remains to show that the number of regluings on the tropical side is the same as the number of choices of $\alpha$ on the factorisation side. The difficulty lies in the fact that there is no direct correspondence between factorisations and tropical covers. Still, this was proved to be true in \cite{Hahn_2022}.
    
    Finally, the $G$--weight contribution of $(\tau_1,\dots,\tau_r)$ is the same for $N_{g,d}^{G}$ and $H_g^G(\mu,\mu)$. Thus, the proof immediately generalises to the case of weighted elliptic Hurwitz numbers.
\end{proof}

In order to explore quasimodularity, we need the following definition.

\begin{definition}
    We fix a combinatorial type $C$ of a tropical curve $\Gamma$ with $l \defeq |V (\Gamma) | \leq 2g(\Gamma) -2$. We further fix an orientation on $E_{trop}$ and a linear ordering $\Omega$ on the vertices of $C$. Denote by $v_i$ the $i$-th vertex according to $\Omega$. We choose points $p_1, \dots, p_{2g-2}$ on $E_{trop}$ linearly ordered with respect to the given orientation of $E_{trop}$. Then fix a series of integers $\mathbf{g}' = (g_1, \dots, g_l)$.

Denote by $\Gamma(C; \Omega, g, \mathbf{g}')$ the set of all covers $\pi \in \Gamma(E_{trop};g,d)$ for some $d \in \N$ that satisfies;
\begin{enumerate}[label=(\roman*)]
    \item for $\pi : \Gamma \to E_{trop}$, $\Gamma$ has combinatorial type $C$,
    \item $\pi(v_i) = p_i$,
    $g(v_i) = g_i$.
\end{enumerate}
We associate a generating series to each combinatorial type $C$;
\[
I_{G,\mathbf{g}'}^{C, \Omega} \defeq \sum_{\pi \in \Gamma(C; \Omega , g, \mathbf{g}')} \operatorname{mult}_G(\pi) q^{deg(\pi)}.
\]
\end{definition}
The correspondence theorem of this section then gives that
\begin{equation}
N_{d,g}^{G,\circ} = \sum_{C, \Omega, \mathbf{g}'}  I_{G,\mathbf{g}'}^{C, \Omega}
\end{equation}
where we sum over all combinatorial types $C$ on at most $l \leq 2g-2$ vertices, order $\Omega$ on $C$ and tuples $\mathbf{g}' = (g_1, \dots, g_l)$.

We are finally ready to prove the refined quasimodularity statement, i.e. that the contribution of each combinatorial type with fixed order is a quasimodular form.

\begin{theorem}\label{thm-refinquasi}
For $|\mathbf{g}'| \geq 2$, the series $I_{G,\mathbf{g}'}^{C, \Omega}
$ is a quasimodular form of mixed weight $\leq 2\left( \sum_{i=1}^l g_i + |E(C)|\right)$.    
\end{theorem}
\begin{proof}
    This is due to \cite[Theorem 6.2]{goujard2017countingfeynmanlikegraphsquasimodularity} 
    where it is proved that the generating series associated to a tropical cover, targetting a curve of fixed combinatorial type, order, and fixed ramification profile whenever the multiplicity of the cover is a polynomial in the edge weights. We need to verify that local vertex multiplicities are polynomial, which is true due to \cref{polynomialgw}.
\end{proof}

We can recover our original quasimodularity statement as a corollary.

\begin{corollary}\label{thirdmainthingsomehow}
The generating series (for $g \geq 2$)
\[
\sum_{d \geq 1} N_{g,d}^{G,\circ} q^d
\]
is a quasimodular form of mixed weight $\leq 6g-6$.
\end{corollary}

\begin{proof}
    The only thing we need to prove here is the claimed weight. For this, consider $2\left( \sum_{i=1}^l g_i + |E(C)|\right)$, where by the handshake lemma, we know $2|E(C)|= \sum_{i=1}^l \val (v_i)$, summing over all vertices of $C$. By construction, $\val (v_i) = \lambda_i - 2 g_i +2 $ so that;
    \begin{align}
    2\left( \sum_{i=1}^l g_i + |E(C)|\right) &= 2 \sum_{i=1}^l g_i + \sum_{i=1}^r (\lambda_i  - 2g_i +2) \\
    &= \sum_{i=1}^r \lambda_i + 2l = 2g-2 +2l \leq 2g-2 +2(2g-2) = 6g-6,
    \end{align}
    coming from the relations $\sum_{i=1}^r \lambda_i = 2g-2$ and $d \leq 2g-2$. This yields the desired result.
\end{proof}

\subsection{Feynman diagrams}
In this section, we express $N_{g,d}^{G,\circ}$ in terms of Feynman diagrams. 

\begin{definition}
    Fix $n>1$. A Feynman diagram is a (non-metrised) graph $\Gamma$ without ends with $n$ vertices labeled $x_1,\dots,x_n$ and $r$ edges labeled $q_1,\dots,q_r$. We will assume that the $q_1,\dots,q_s$ are loop edges, whereas $q_{s+1},\dots, q_r$ are not.
\end{definition}

Next, we define a propagators and Feynman integrals.

\begin{definition}
    Fix a non--negative integer $g$. Let $\Gamma$ be a Feynman diagram with $n$ vertices, $\Omega$ an order on the vertices and choose $\lambda=(\lambda_1,\dots,\lambda_n)$ a partition of $2g-2$. Moreover, define numbers $g_i$ as $\lambda_i=2g_i-2+\mathrm{val}(x_i)$.
    
    For $k>s$ denote the edges adjacent to the non--loop edge $q_k$ by $x_{k^1}$ and $x_{k^2}$, where $k_1<k_2$ in $\Omega$. We then define a \textbf{propagator} introducing new variables $z_1,\dots,z_n$.

    \begin{align}
        P\left(\frac{x_{k^1}}{x_{k^2}},q_k\right)=&\sum_{w=1}^\infty\mathcal{S}(wz_{k^1})\mathcal{S}(wz_{k^2})w\left(\frac{x_{k^1}}{x_{k^2}}\right)^w\\
        +&\sum_{a=0}^\infty\left(\sum_{w\mid a}\mathcal{S}(wz_{k^1})\mathcal{S}(wz_{k^2})w\left(\left(\frac{x_{k^1}}{x_{k^2}}\right)^w+\left(\frac{x_{k^2}}{x_{k^1}}\right)^w\right)\right)q_k^a.
    \end{align}
    For loop edges $q_k$ with vertex $x_{k^1}$, we define
    \begin{equation}
        P^{\textrm{loop}}(q_k)=\sum_{a=1}^\infty\left(\sum_{w\mid a}\mathcal{S}(wz_{k^1})^2w\right)q_k^a.
    \end{equation}
    We define the refined Feynman integral associated to $\Gamma$, $\Omega$ and $\lambda$ as
    \begin{equation}
        I_{\Gamma,\Omega,\lambda}(q_1,\dots,q_k)=[z_1^{2g_1}\cdots z_n^{2g_n}][x_1^{0}\cdots x_n^{0}]\prod_{i=1}^n\frac{1}{\mathcal{S}(z_i)^2}\prod_{i=1}^sP^{\textrm{loop}}(q_k)\prod_{i=s+1}^nP\left(\frac{x_{k^1}}{x_{k^2}},q_k\right)
    \end{equation}
    and the Feynman integral as
    \begin{equation}
        I_{\Gamma,\Omega,\lambda}(q)=[z_1^{2g_1}\cdots z_n^{2g_n}][x_1^{0}\cdots x_n^{0}]\prod_{i=1}^n\frac{1}{\mathcal{S}(z_i)^2}\prod_{i=1}^sP^{\textrm{loop}}(q)\prod_{i=s+1}^nP\left(\frac{x_{k^1}}{x_{k^2}},q\right)
    \end{equation}
    Finally, we set
    \begin{equation}
        I_{\Gamma,\lambda}(q_1,\dots,q_k)=\sum_{\Omega}I_{\Gamma,\Omega,\lambda}(q_1,\dots,q_k)
    \end{equation}
    and
    \begin{equation}
        I_{\Gamma,\lambda}(q)=\sum_{\Omega}I_{\Gamma,\Omega,\lambda}(q),
    \end{equation}
    where each sum is over all orders on the vertices of $\Gamma$.   
\end{definition}

The main result of this subsection is as follows.

\begin{theorem}
\label{thm-feynman}
    Let $g\ge2$ be an integer, then we have

    \begin{equation}
        \sum N_{g,d}^{G,\circ}q^d=\sum_{(\Gamma,\lambda)}\frac{\prod_i \lambda_i!A_{\lambda_i}}{|(\mathrm{Aut}(\Gamma),\lambda)|}I_{\Gamma,\lambda}(q),
    \end{equation}
    where $|(\mathrm{Aut}(\Gamma),\lambda)|$ denotes the number of automorphisms of $\Gamma$ as a graph respecting the labelling of the vertices by $\lambda$.
\end{theorem}

\begin{proof}
   Expanding $I_{\Gamma,\Omega,\lambda}(q_1,\dots,q_k)$, a summand looks like
   \begin{equation}
      \frac{\prod_{k=1}^s\mathcal{S}(w_kz_{k^1})^2\prod_{k=s+1}^r\mathcal{S}(w_kz_{k^1})\mathcal{S}(w_kz_{k^2})}{\prod_{k=1}^{n}\mathcal{S}(z_i)^2}\prod_{k=1}^rw_k\left(\frac{x_{k^i}}{x_{k^j}}\right)^{w_k}q_k^{a_k},
   \end{equation}

    We now associate a tropical cover to such a summand. We start with the Feynman graph $\Gamma$ and the order $\Omega$. We map the vertex $x_i$ of $\Gamma$ to the $i$--th marked point $p_i$ on $E_{\textrm{trop}}$. The summand encodes how to map the edges, as well as the multiplicity of the constructed tropical cover.
    
    For
    \begin{equation}
        w_k\left(\frac{x_{k^1}}{x_{k^2}}\right)^{w_k}q_k^{a_k},
    \end{equation}
    we first assume $a_k=0$. In that case, we give the edge $q_k$ weight $w_k$ and simply map it to the interval in $E_{\textrm{trop}}$ from $p_{k^i}$ to $p_{k^j}$. If $a_k>0$, we instead continue the edge so that it passes through the fibre over $p_0$ exactly $a_k$ many times before it connects to $x_{k^j}$.
    Proceeding like this for all edges, we obtain a map $\Gamma\to E_{\textrm{trop}}$ with prescribed edge weights. In order for it to be a tropical cover, have to check the balancing condition. Indeed, we have for each vertex $x_k$ with outgoing edge $q_i$, $x_k$ and $q_i$ appear in the product with $x_k^{\omega(q_i)}$ in the numerator. On the other hand, for an incoming edge $q_i$, we have that $x_k^{\omega(q_i)}$ appears in the denominator. Thus, taking the coefficient of $x_k^0$ means exactly that the sum of incoming edges equal the sum of out going edges. This is the balancing condition.

    Finally, we have to check the multiplicities. The factor $w_k$ takes care of the edge weights. Recall the vertex multiplicity 
    \begin{equation}
        \lambda_i! | \hspace{0em}\aut(\mathbf{x}^+) \hspace{0em}| | \hspace{0em} \aut(\mathbf{x}^-) \hspace{0em}| A_{\lambda_i} \Big< \tau_{2g_2^i -2}(\omega)  \Big>_{g_2^i}^{\rsph, \circ} \Big< \mathbf{x}^{+}  \mid \tau_{2g_1^i -2+ \ell(\mathbf{x}^{+}) + \ell(\mathbf{x}^{-})}(\omega) \hspace{0em}\mid \hspace{0em} \mathbf{x}^{-}  \hspace{0em} \Big>_{g_1^i}^{\rsph, \circ}.
    \end{equation}

    The factors $A_{\lambda_i}\lambda_i!$ are global factors in the statement. The only remaining aspect are the automorphisms, as well as the split of $g_i$ into $g_i=g_1^i+g_2^i$. This arises as follows. We observe that 

    \begin{equation}
        \frac{1}{\mathcal{S}(z)}=\sum_g\Big< \tau_{2g -2}(\omega)  \Big>_{g_2^i}^{\rsph, \circ}z^{2g},
    \end{equation}

    while we have already seen in \cref{polynomialgw} that

    \begin{equation}
        | \hspace{0em}\aut(\mathbf{x}^+) \hspace{0em}| | \hspace{0em} \aut(\mathbf{x}^-) \hspace{0em}|\Big< \mathbf{x}^{+}  \mid \tau_{2g -2+ \ell(\mathbf{x}^{+}) + \ell(\mathbf{x}^{-})}(\omega) \hspace{0em}\mid \hspace{0em} \mathbf{x}^{-}  \hspace{0em} \Big>_{g}^{\rsph, \circ}=[z^{2g}]\frac{\prod\mathcal{S}(x_iz)}{\mathcal{S}(z)}.
    \end{equation}
For each vertex $x_i$, we have a product

    \begin{equation}
        \frac{\prod_k\mathcal{S}(\omega(q_k)z_i)}{\mathcal{S}(z_i)^2}=\frac{1}{\mathcal{S}(z_i)}\frac{\prod_k\mathcal{S}(\omega(q_k)z_i)}{\mathcal{S}(z_i)},
    \end{equation}
    where the product over $k$ runs over all edges $q_k$ adjacent to $x_i$.

    We now observe that

    \begin{equation}
        [z_i^{2g}]\frac{1}{\mathcal{S}(z_i)}\frac{\prod_k\mathcal{S}(\omega(q_k)z_i)}{\mathcal{S}(z_i)}=\sum_{g^1+g^2=g}[z^{2g^1}]\frac{1}{\mathcal{S}(z_i)}[z^{2g^2}]\frac{\prod_k\mathcal{S}(\omega(q_k)z_i)}{\mathcal{S}(z_i)}.
    \end{equation}

    Thus, we have proved that $\frac{\prod_i \lambda_i!A_{\lambda_i}}{|(\mathrm{Aut}(\Gamma),\lambda)|}I_{\Gamma,\lambda}(q)$ is the generating function of contributions of all tropical covers with source curve $\Gamma$ respecting the order $\Omega$ and partition $\lambda$, while all edges are labelled. 

    By moving to $I_{\Gamma,\lambda}(q)$ we sum over all vertex orderings and drop the edge labels. Thus, we need to divide by the graph automorphisms of $\Gamma$ that however to respect $\lambda$. This is encoded on the left hand side in the automorphisms of the tropical covers.
    This proves the result.
    
\end{proof}

\printbibliography

\end{document}